\documentclass[a4paper,10pt,reqno]{amsart}

\usepackage{amssymb,amsmath,amsfonts,amsthm,mathtools,color}
\usepackage{latexsym,mathrsfs,pb-diagram}
\usepackage[pdftex]{graphicx}
\usepackage[all]{xy}
\usepackage[hmargin=2.5cm,vmargin=3.5cm]{geometry}
\usepackage[plainpages=false,colorlinks,pdfpagelabels]{hyperref}
\hypersetup{
  urlcolor=black,
  citecolor=green,
  linkcolor=blue
}

\numberwithin{equation}{section}

\theoremstyle{definition}
\newtheorem{Def}{Definition}[section]

\theoremstyle{remark}
\newtheorem{Exa}[Def]{Example}
\newtheorem{Rem}[Def]{Remark}

\theoremstyle{plain}
\newtheorem{Prop}[Def]{Proposition}
\newtheorem{Cor}[Def]{Corollary}
\newtheorem{Thm}[Def]{Theorem}
\newtheorem{Lem}[Def]{Lemma}

\newcommand{\dfn}{\mathrel{\dot{=}}}

\newcommand{\st}{ \ ; \ }
\newcommand{\rarr}{\rightarrow}
\newcommand{\sset}{\subset}

\newcommand{\Z}{\mathbb{Z}}
\newcommand{\N}{\mathbb{N}}
\newcommand{\R}{\mathbb{R}}

\newcommand{\C}{\mathbb{C}}

\newcommand{\TR}[5]{\begin{array}{c c c c c}
    {#1} & : & {#3} & \longrightarrow & {#5}\\
    & & {#2} & \longmapsto & {#4}
  \end{array}
}

\newcommand{\transp}[1]{\prescript{\mathrm{t} \!}{}{{#1}}}
\newcommand{\PT}{\prescript{\mathrm{t} \!}{}{P}}
\newcommand{\ort}{\mathrm{o}}
\DeclareMathOperator{\Span}{\mathrm{span}}
\DeclareMathOperator{\ran}{\mathrm{ran}}

\newcommand{\del}{\partial}
\newcommand{\dd}{\mathrm{d}}

\newcommand{\D}{\mathscr{D}}

\newcommand{\cinfty}{\mathscr{C}^\infty}

\newcommand{\sob}{\mathscr{H}}

\newcommand{\LL}{\mathrm{L}}

\DeclareMathOperator{\lie}{\mathrm{Lie}}

\DeclareMathOperator{\SU}{\mathrm{SU}}

\newcommand{\TT}{\mathbb{T}}

\newcommand{\gr}[1]{\mathfrak{#1}}

\DeclareMathOperator{\ad}{\mathrm{ad}}
\newcommand{\vv}[1]{\mathrm{#1}}

\author{Gabriel Ara\'{u}jo}
\address{Universidade de S{\~a}o Paulo, ICMC-USP, S{\~a}o Carlos, SP, Brazil}
\email{\texttt{gccsa@icmc.usp.br}}

\author{Igor A.~Ferra}
\address{Universidade Federal do ABC, CMCC-UFABC, S{\~a}o Bernardo do Campo, SP, Brazil}
\email{\texttt{ferra.igor@ufabc.edu.br}}

\author{Luis F.~Ragognette}
\address{Universidade Federal de S{\~a}o Carlos, DM-UFSCar, S{\~a}o Carlos, SP, Brazil}
\email{\texttt{luisragognette@dm.ufscar.br}}

\thanks{This work was supported by the S{\~a}o Paulo Research Foundation (FAPESP, grants~2016/13620-5 and~2018/12273-5).}

\title[]{Global solvability and propagation of regularity of sums of squares on compact manifolds}

\begin{document}

\begin{abstract} We investigate global solvability, in the framework of smooth functions and Schwartz distributions, of certain sums of squares of vector fields defined on a product of compact Riemannian manifolds $T \times G$, where $G$ is further assumed to be a Lie group. As in a recent article due to the authors, our analysis is carried out in terms of a system of left-invariant vector fields on $G$ naturally associated with the operator under study, a simpler object which nevertheless conveys enough information about the original operator so as to fully encode its solvability. As a welcome side effect of the tools developed for our main purpose, we easily prove a general result on propagation of regularity for such operators.
\end{abstract}

\maketitle


\section{Introduction}

In this work we present some results concerning the global solvability of linear differential operators on compact manifolds. Inspired by previous investigations of similar questions on tori~\cite{petronilho02}, our main result here characterizes global solvability for a class of sums of squares of vector fields in a product manifold $T\times G$, where $T$ will be assumed a Riemannian manifold and $G$ a Lie group, both compact, connected and oriented. More precisely, we consider operators of the following sort
\begin{align}
  P &\dfn \Delta_T - \sum_{\ell = 1}^N \Big(\sum_{j = 1}^m a_{\ell j}(t) \vv{X}_j + \vv{W}_\ell \Big)^2 \label{eq:Pderf_intro}
\end{align}
where: $\Delta_T$ is the Laplace-Beltrami operator associated with the underlying metric on $T$ (acting on functions); $a_{\ell j}$ are smooth, real-valued functions on $T$; $\vv{W}_\ell$ are skew-symmetric real vector fields in $T$; and $\vv{X}_1, \ldots, \vv{X}_m \in \gr{g}$ is a basis of left-invariant vector fields in $G$ (as usual $\gr{g}$ denotes the Lie algebra of $G$, and corresponds to the space of real vector fields with constant coefficients when $G$ is the $m$-dimensional torus $\mathbb{T}^m$).

The present article is thus a natural continuation of~\cite{afr20}, in which the authors classified the global hypoellipticity of $P$ in $T \times G$ in terms of the global hypoellipticity of a system of left-invariant vector fields $\mathcal{L}$ in $G$. Namely, an $\vv{L} \in \gr{g}$ belongs to $\mathcal{L}$ if and only if there are $\ell\in \{1, \ldots, N\}$ and $t\in T$ such that
\begin{align*}
  \mathrm{L} &= \sum_{j = 1}^m a_{\ell j}(t) \vv{X}_j.
\end{align*}
Here, the connection between global properties of the  operator with variable coefficients $P$ on $T \times G$ with those of the system of vector fields with ``constant coefficients'' $\mathcal{L}$ on $G$ -- a somewhat simpler object -- is further explored. Our Theorem~\ref{Thm:main-P} relates their global solvability properties, and generalizes previous results~\cite[Theorem 1.1]{petronilho02} in which the global solvability of
\begin{align*}
  P &\dfn -\sum_{k = 1}^n \partial_{t_k}^2 - \Big(\sum_{j=1}^m a_j(t)\partial_{x_j} \Big)^2  
\end{align*}
is fully characterized on $\TT^n_t \times \TT^m_x$ by means of Diophantine conditions. The relationship between Petronilho's conditions and ours is made explicit in Section~\ref{sec:examples}, where we specialize our discussion to the commutative case.

It should be pointed out that in order to attain all the equivalences in Theorem~\ref{Thm:main-P} we require $\mathcal{L}$ to satisfy an extra, technical condition~\eqref{it:thm15_hyp1}, which already appeared in~\cite{afr20}: roughly speaking, it says that the linear span of $\mathcal{L}$ can be decomposed into some commutative ``blocks'' in a precise manner. While such a property automatically holds on a torus, it is still flexible enough in the non-commutative setting so as to allow us to exhibit a myriad of positive examples. The necessity of this assumption, either here or in~\cite{afr20}, is still an open problem though.

Although Definition~\ref{def:solvability_compact} brings in three different notions of global solvability, Theorem~\ref{Thm:main-P} ensures that they are all equivalent for our $P$. We go further and give one more: a weakened version of global hypoellipticity (``modulo $\ker P$'') which we call, in a non-standard fashion, \emph{almost global hypoellipticity} -- $\mathrm{(AGH)}$ for short; see Definition~\ref{Def:AGH}. The idea of quotienting out the kernel of $P$ is very natural and not entirely new, and seems to date back in some form or another at least to the works of H{\"o}rmander on operators of simple real characteristics. We then take the opportunity to develop this notion and its relations to global solvability in a much broader context.

It is proverbial that global hypoellipticity of a LPDO $Q$ on a compact manifold $M$ implies global solvability of its transpose $\transp{Q}$ (see~\cite{treves_tvs} for the classical local analogue). But there are examples in the literature that show classes of operators whose global hypoellipticity imply their own global solvability e.g.~for certain vector fields on a torus~\cite[Corollary~3.1]{hounie79} \cite[Theorem~3.1]{petronilho11}, and when there exists a suitable elliptic operator that commutes with $Q$~\cite[Theorem~13]{gw73}, to name a few. Here, by using the notion of almost global hypoellipticity, we prove in Theorem~\ref{thm:general_agh_closedrange} that in particular this fact holds true in general, i.e., that every globally hypoelliptic operator in $M$ is globally solvable (meaning here that it has closed range in $\cinfty(M)$). Actually, the operator $Q$ does not even need to be a differential one, as only mild continuity properties are required in that proof. Of course the converse is false in general since there are globally solvable operators whose kernel in $\D'(M)$ is infinite dimensional (which cannot happen if the said kernel is contained in a fixed Sobolev space), for instance $P \dfn D_t+D_x$ in $\mathbb T^2$ -- precisely the phenomenon that almost global hypoellipticity aims to accommodate. Therefore, this result shows another big difference when dealing with global notions instead of local ones, since the Kannai operator~\cite{kannai71} is locally hypoelliptic operator yet not locally solvable.

While the general converse to Theorem~\ref{thm:general_agh_closedrange} is false (Example~\ref{exa:noconverseAGH}), in Corollary~\ref{prop:gs_agh} we devise a partial one assuming an extra hypothesis about density in the kernel of the transpose operator. Since our operator $P$~\eqref{eq:Pderf_intro} satisfies this additional hypothesis, we were able to prove that its global solvability is in fact equivalent to its almost global hypoellipticity. Another fascinating aspect of this notion in connection with our main line of investigation is that although a non-Abelian compact Lie group cannot carry real left-invariant vector fields which are globally hypoelliptic~\cite{gw73b} (in sharp contrast with the situation one finds in the torus~\cite{gw72}), one does not have to dig too deep into the non-commutative realm to find left-invariant vector fields that are~$\mathrm{(AGH)}$ (Example~\ref{exa:su2}).

As a second application of the theoretical framework developed here and in~\cite{afr20}, we easily prove a result (Theorem~\ref{thm:propagation}) on propagation of regularity for~\eqref{eq:Pderf_intro} in the same vein as~\cite[Theorem~1]{hp02}, but much more general -- it does not even require the technical assumption~\eqref{it:thm15_hyp1}.

This work is divided as follows. In Section~\ref{sec:prel} we introduce the machinery and notation that will be used throughout the paper, eventually referring the reader to~\cite{afr20} for more details. Our abstract results are proved in Section~\ref{sec:gs-gen}, and in Section~\ref{sec:systems_AGH} we discuss some results about almost global hypoellipticity of systems of vector fields, with particular emphasis on left-invariant vector fields on Lie groups. Sections~\ref{Sec:tube-type} and~\ref{sec:charac-P} are then devoted to address the characterization of the global solvability of $P$, and a wide class of examples is obtained through the results presented in Section~\ref{sec:examples}, where we translate the previous ones into Diophantine conditions on the torus. We end this paper with a result about propagation of regularity in Section~\ref{sec:singularities}.

\section{Preliminaries}  \label{sec:prel}

Let $M$ be a compact, connected, oriented smooth manifold, for simplicity carrying a Riemannian metric: its underlying volume form $\dd V_M$ allows us to regard $\D'(M)$, the space of Schwartz distributions on $M$, as the topological dual of $\cinfty(M)$. We denote by $\Delta_M$ the Laplace-Beltrami operator associated with our metric, by $\sigma(\Delta_M)$ its spectrum and by $E^M_\lambda$ the eigenspace associate with $\lambda \in \sigma(\Delta_M)$. We then define the Fourier projection map $\mathcal{F}^{M}_{\lambda}: \D'(M) \to E^M_\lambda$ as follows: given $f \in \D'(M)$ take $\mathcal{F}^M_\lambda(f)$ as the unique element in $E_\lambda^M$ satisfying
\begin{align*}
  \langle \mathcal{F}^M_\lambda(f), \phi \rangle_{L^2(M)} &= \langle f, \bar{\phi} \rangle, \quad \forall \phi \in E^M_\lambda.
\end{align*}
One can then characterize smooth functions and distributions in terms of the decay of their projections:
\begin{Prop} \label{prop:charac_smoothness_fourier_proj} For a sequence $a = (a(\lambda))_{\lambda \in \sigma(\Delta_M)}$ where $a(\lambda) \in E_\lambda^M$ for all $\lambda \in \sigma(\Delta_M)$ the following characterizations hold:
  \begin{enumerate}
  \item $a = \mathcal{F}^M(f)$ for some $f \in \cinfty(M)$ if and only if for every $s > 0$ there exists $C > 0$ such that
    \begin{align*}
      \| a(\lambda) \|_{L^2(M)} &\leq C (1 + \lambda)^{-s}, \quad \forall \lambda \in \sigma(\Delta_M).
    \end{align*}
  \item $a = \mathcal{F}^M(f)$ for some $f \in \D'(M)$ if and only if there exist $C, s > 0$ such that
    \begin{align*}
      \| a(\lambda) \|_{L^2(M)} &\leq C (1 + \lambda)^{s}, \quad \forall \lambda \in \sigma(\Delta_M).
    \end{align*}
  \end{enumerate}
\end{Prop}
For more details see e.g.~\cite{araujo19}. We also recall how the standard topologies on $\cinfty(M)$ and $\D'(M)$ may be given in terms of a suitable scale of Sobolev spaces
\begin{align*}
  \sob^k(M) &\dfn \left\{ f \in L^2(M) \st (I + \Delta_M)^k f \in L^2(M) \right\}, \quad k \in \Z_+.
\end{align*}
Each $\sob^k(M)$ carries a natural Banach space topology with norm
\begin{align}
  \|f \|_{\sob^k(M)} &\dfn \|(I + \Delta_M)^k f \|_{L^2(M)}, \label{eq:total_sobolev_norm}
\end{align}
which turns the inclusion $\sob^{k + 1} (M) \hookrightarrow \sob^k(M)$ into a compact map: the topology of $\cinfty(M)$ is then precisely the locally convex projective limit topology induced by the family $\{\sob^k(M)\}_{k \in \Z_+}$, and by further defining $\sob^{-k}(M)$ as the dual of $\sob^{k}(M)$ for every $k \in \Z_+$ one can prove that the strong dual topology on $\D'(M) = \cinfty(M)'$ coincides with the injective limit topology induced by $\{\sob^{-k}(M)\}_{k \in \Z_+}$ -- and this will always be the topology we shall endow $\D'(M)$ with. In particular, $\cinfty(M)$ and $\D'(M)$ are a FS and a DFS space, respectively, each one the strong dual of the other~\cite{komatsu67}.

\subsection{Partial Fourier expansions} \label{sec:partial_exp}

Following~\cite[Section~2]{afr20}, given $T, G$ two compact, connected and oriented Riemannian manifolds, whose dimensions will be denoted by $n \dfn \dim T$ and $m \dfn \dim G$, one may appeal to the product structure of $M \dfn T \times G$ (always endowed with the product metric) to give a finer description of the constructions above. The proposition below summarizes the main properties we will be implicitly using. On time, if $f$ is a function or distribution on $G$ we will also regard it as an object defined on $M$ which does not depend on the first variable (and analogously for distributions on $T$) without further notice; when clarity demands we will explicitly write it as $1_T \otimes f$ where $1_T$ stands for the constant function on $T$. Along the same lines a differential operator $P$ defined on either factor will be understood as an operator on $M$, lifted in the obvious fashion.
\begin{Prop} \label{prop:relationship_LBs} \hfill
  \begin{enumerate}
  \item \label{it:delta12} $\Delta_M = \Delta_T + \Delta_G$ as differential operators on $T \times G$.
  \item \label{it:hilbert_basis_product} For each $\mu \in \sigma(\Delta_T)$ and $\lambda \in \sigma(\Delta_G)$ let
    \begin{itemize}
    \item $\{ \psi^\mu_i \st 1 \leq i \leq d^T_\mu \}$ be a basis for $E^T_\mu$, orthonormal w.r.t.~the $L^2(T)$ inner product, and
    \item $\{ \phi^\lambda_j \st 1 \leq j \leq d^G_\lambda \}$ be a basis for $E^G_\lambda$, orthonormal w.r.t.~the $L^2(G)$ inner product.
    \end{itemize}
    Then the set
    \begin{align*}
      \mathcal{S} &\dfn \{ \psi^\mu_i \otimes \phi^\lambda_j \st 1 \leq i \leq d^T_\mu, \ 1 \leq j \leq d^G_\lambda, \ \mu \in \sigma(\Delta_T), \ \lambda \in \sigma(\Delta_G) \}
    \end{align*}
    is a Hilbert basis for $L^2(T \times G)$.
  \item \label{it:alphaissum} Every $\alpha \in \sigma(\Delta_M)$ is of the form $\alpha = \mu + \lambda$ for some $\mu \in \sigma(\Delta_T)$ and $\lambda \in \sigma(\Delta_G)$.
  \item \label{it:relationshipeigens} If for each $\alpha \in \R_+$ we define the set
    \begin{align*}
      \mathcal{P}(\alpha) &\dfn \{ (\mu, \lambda) \in \sigma(\Delta_T) \times \sigma(\Delta_G) \st \mu + \lambda = \alpha \} 
    \end{align*}
    -- which is finite -- then the eigenspace of $\Delta_M$ associated with $\alpha \in \sigma(\Delta_M)$ is precisely
    \begin{align*}
      E_\alpha^M &= \bigoplus_{(\mu, \lambda) \in \mathcal{P}(\alpha)} E^T_{\mu} \otimes E^G_{\lambda}
    \end{align*}
    and an orthonormal basis for this space w.r.t.~the $L^2(T \times G)$ inner product is
    \begin{align*}
      \{ \psi^\mu_i \otimes \phi^\lambda_j  \st 1 \leq i \leq d^T_\mu, \ 1 \leq j \leq d^G_\lambda, \ (\mu, \lambda) \in \mathcal{P}(\alpha) \}. 
    \end{align*}
  \end{enumerate}
\end{Prop}
Still following~\cite{afr20}, we can consider ``partial'' Fourier expansions in $T \times G$. If $f \in \D'(T \times G)$ and $\lambda \in \sigma(\Delta_G)$ then $\mathcal F_\lambda^G (f) \in \D'(T;E_\lambda^G) \cong \D'(T)\otimes E_\lambda^G$ is characterized by the following property: given any orthonormal basis $\{\phi_1^\lambda,\ldots,\phi_{d_\lambda^G}^\lambda\}$ of $E_\lambda^G$ we can write
\begin{align*}
  \mathcal{F}_\lambda^G (f) &= \sum_{i=1}^{d_\lambda^G} \mathcal{F}_\lambda^G(f)_i \otimes \phi_i^\lambda,
\end{align*}
where $\mathcal F_\lambda^G(f)_i \in \D'(T)$ is defined by
\begin{align*}
  \langle \mathcal F_\lambda^G(f)_i, \psi \rangle &\dfn \langle f, \psi \otimes \overline{\phi^\lambda_i} \rangle, \quad \forall \psi \in \cinfty(T).
\end{align*}
The ``total'' Fourier projection of $f$ can then be recovered from the partial ones as
\begin{align}
  \mathcal{F}^{M}_\alpha(f) &= \sum_{\mu+\lambda=\alpha} \mathcal{F}^{T}_\mu \mathcal{F}^G_\lambda(f), \quad \alpha \in \sigma(\Delta_M), \label{eq:total_from_partials}
\end{align}
thanks to Proposition~\ref{prop:relationship_LBs} -- actually, any $f \in \D'(T \times G)$ can be written in terms of the orthonormal basis~$\mathcal{S}$ as
\begin{align}
  f &= \sum_{\alpha \in \sigma(\Delta_M)} \sum_{\mu + \lambda = \alpha} \sum_{j=1}^{d^T_\mu} \sum_{i=1}^{d_\lambda^G} \left\langle f, \overline{\psi^\mu_j\otimes \phi^\lambda_i} \right\rangle \psi^\mu_j\otimes \phi^\lambda_i \label{eq:fourier_in_bases}
\end{align}
with convergence in $\D'(T \times G)$ (resp.~in $\cinfty(T \times G)$, provided $f$ is smooth). Appealing to Weyl's asymptotic estimates~\cite[p.~155]{chavel_eigenvalues} and Proposition~\ref{prop:relationship_LBs} one then translates Proposition~\ref{prop:charac_smoothness_fourier_proj} as follows:
\begin{Cor} \label{cor:charac_smoothness_fourier_proj} A formal series
  \begin{align*}
    \sum_{\alpha \in \sigma(\Delta_M)} \sum_{\mu + \lambda = \alpha} \sum_{j=1}^{d^T_\mu} \sum_{i=1}^{d_\lambda^G} f_{ij}^{\mu \lambda} \ \psi^\mu_j\otimes \phi^\lambda_i, \quad f_{ij}^{\mu \lambda} \in \C,
  \end{align*}
  represents, in the sense of~\eqref{eq:fourier_in_bases}:
  \begin{enumerate}
  \item an $f \in \cinfty(T \times G)$ if and only if for every $s > 0$ there exists $C > 0$ such that
    \begin{align*}
      | f_{ij}^{\mu \lambda} | &\leq C (1 + \mu + \lambda)^{-s}, \quad \forall i, j, \mu, \lambda;
    \end{align*}
  \item an $f \in \D'(T \times G)$ if and only if there exist $C, s > 0$ such that
    \begin{align*}
      | f_{ij}^{\mu \lambda} | &\leq C (1 + \mu + \lambda)^{s}, \quad \forall i, j, \mu, \lambda.
    \end{align*}
  \end{enumerate}
  In both cases, the series~\eqref{eq:fourier_in_bases} converges in the appropriate topology.
\end{Cor}

The digression above leads one to consider families of ``mixed Sobolev norms'' on $T \times G$, namely
\begin{align*}
  \| f \|_{\sob^{j, k}(T \times G)} &\dfn \|(1+\Delta_T)^{j} (1+ \Delta_G)^{k} f \|_{L^2(T \times G)} 
\end{align*}
where $j, k \in \Z_+$, which behave better with respect to the partial Fourier projections while being equivalent, as far as the topology of $\cinfty(T \times G)$ is concerned, to those determined by $\Delta_M$~\eqref{eq:total_sobolev_norm}. To see that, recall that for every $j, k\in \Z_+$ we have for $f$ a smooth function
\begin{align*}
  \|(1+\Delta_M)^{k} f\|_{L^2(T \times G)}^2 &= \sum_{\alpha \in \sigma(\Delta_M)} (1+\alpha)^{2k} \|\mathcal{F}^{M}_\alpha(f)\|^2_{L^2(T \times G)}
\end{align*}
and
\begin{align}
  \|(1 + \Delta_T)^{j} (1+ \Delta_G)^{k} f \|_{L^2(T \times G)}^2 &= \sum_{\alpha\in \sigma(\Delta_M)} \sum_{\mu+\lambda=\alpha}(1+\mu)^{2j}(1+\lambda)^{2k} \|\mathcal{F}^{T}_\mu \mathcal{F}^G_\lambda(f)\|^2_{L^2(T \times G)} \label{eq:mixed_sobolev_open}
\end{align}
where $\mu\in \sigma(\Delta_T)$, $\lambda\in \sigma(\Delta_G)$. Using the inequalities that follow from the positivity of the eigenvalues
\begin{align*}
  1+\alpha \leq (1+\mu)(1+\lambda) \leq (1+\alpha)^2
\end{align*}
and that the terms in the finite sum~\eqref{eq:total_from_partials} are orthogonal in $L^2(T \times G)$ we conclude that
\begin{align} \label{eq:sobolev_comparation}
  \| f \|_{\sob^{k}(T \times G)} \leq \| f \|_{\sob^{k,k}(T \times G)}, &\quad \| f \|_{\sob^{j, k}(T \times G)} \leq \| f \|_{\sob^{j + k}(T \times G)}, \quad \forall j, k \in \Z_+.
\end{align}
We have thus proved:
\begin{Lem} \label{lem:topology_mixed} The topology of $\cinfty (T \times G)$ is generated by the norms $\| \cdot \|_{\sob^{j,k}(T\times G)}$.
\end{Lem}
The next result characterizes those distributions on $T \times G$ which do not depend on one of the factors.
\begin{Lem} \label{Lem:dist-const-T} Let $f \in \D'(T\times G)$. Then $f$ does not depend on $T$ if and only if $\Delta_T f = 0$.
  \begin{proof} One direction is clear: if $f = 1_T\otimes g$ for some $g \in \D'(G)$, then $\Delta_T f = (\Delta_T 1_T)\otimes g = 0$. For the converse, suppose that $\Delta_T f = 0$ and write $f$ as in~\eqref{eq:fourier_in_bases}: we have the equality
    \begin{align*}
      0 = \Delta_T f = \sum_{\alpha \in \sigma(\Delta_M)}\sum_{\mu + \lambda = \alpha} \sum_{j=1}^{d^T_\mu} \sum_{i=1}^{d_\lambda^G} \mu  \left\langle f, \overline{\psi^\mu_j\otimes \phi^\lambda_i}\right\rangle \psi^\mu_j\otimes \phi^\lambda_i
    \end{align*}
    and the uniqueness of the expansion in Fourier series yields
    \begin{align*}
      \left\langle f, \overline{\psi^\mu_j\otimes \phi^\lambda_i}\right\rangle &= 0, \quad \forall \mu \neq 0, \ \lambda \in \sigma(\Delta_G)
    \end{align*}
    thus implying that the remaining terms in~\eqref{eq:fourier_in_bases} are precisely
    \begin{align*}
      f &=\sum_{\lambda \in \sigma(\Delta_G)} \sum_{i=1}^{d_\lambda^G} \left\langle f, \overline{1_T\otimes \phi^\lambda_i}\right\rangle 1_T\otimes \phi^\lambda_i.
    \end{align*}
    It follows immediately from Corollary~\ref{cor:charac_smoothness_fourier_proj} that
    \begin{align*}
    g &\dfn \sum_{\lambda \in \sigma(\Delta_G)} \sum_{i=1}^{d_\lambda^G} \left\langle f, \overline{1_T\otimes \phi^\lambda_i}\right\rangle \phi^\lambda_i \in \D'(G).
    \end{align*}
    It remains to prove that $f$ equals $w \dfn  1_T\otimes g \in \D'(T\times G)$: we can write
    \begin{align*}
      w &= \sum_{\alpha \in \sigma(\Delta_M)} \sum_{\mu + \lambda = \alpha} \sum_{j=1}^{d^T_\mu} \sum_{i=1}^{d_\lambda^G} \left\langle w, \overline{\psi^\mu_j\otimes \phi^\lambda_i}\right\rangle \psi^\mu_j\otimes \phi^\lambda_i
    \end{align*}
    where
    \begin{align*}
      \left\langle w, \overline{\psi^\mu_j\otimes \phi^\lambda_i}\right\rangle = \left\langle 1_T\otimes g, \overline{\psi^\mu_j\otimes \phi^\lambda_i}\right\rangle = \left\langle 1_T,\overline{\psi^\mu_j}\right\rangle \left\langle g, \overline{\phi^\lambda_i}\right\rangle =
      \begin{cases}
        0, &\text{if $\mu \neq 0$}; \\
        \left\langle g, \overline{\phi^\lambda_i}\right\rangle, &\text{if $\mu = 0$}.
      \end{cases}
    \end{align*}
    We also have, by the definition of $g$,
    \begin{align*}
      \left\langle g, \overline{\phi^\lambda_i}\right\rangle &=  \left\langle f, \overline{1_T\otimes \phi^\lambda_i}\right\rangle, \quad \forall \lambda \in \sigma(\Delta_G), \ i \in \{1,\ldots,d_\lambda^G\}
    \end{align*}
    and therefore we conclude that
    \begin{align*}
    w = \sum_{\alpha \in \sigma(\Delta_M)} \sum_{\mu + \lambda = \alpha} \sum_{i=1}^{d_\lambda^G} \left\langle g, \overline{\phi^\lambda_i}\right\rangle 1_T\otimes \phi^\lambda_i = \sum_{\lambda \in \sigma(\Delta_G)} \sum_{i=1}^{d_\lambda^G}  \left\langle f, \overline{1_T\otimes \phi^\lambda_i}\right\rangle 1_T\otimes \phi^\lambda_i = f.
    \end{align*}
  \end{proof}
\end{Lem}

\subsection{Spectral clusters} \label{sec:spectral-cluesters}

By a \emph{spectral cluster} on $M$ we mean a family $\mathcal{A} \dfn (\mathcal{A}_\lambda)_{\lambda \in \sigma(\Delta_M)}$ such that $\mathcal{A}_\lambda$ is a linear subspace of $E_\lambda^M$ for each $\lambda \in \sigma(\Delta_M)$. Each spectral cluster $\mathcal{A}$ determines an \emph{orthogonal cluster} $\mathcal{A}^\bot \dfn (\mathcal{A}_\lambda^\bot)_{\lambda \in \sigma(\Delta_M)}$ w.r.t.~the $L^2(M)$ inner product on $E_\lambda^M$. We define $\D'_{\mathcal{A}}(M)$ as the space of all $f \in \D'(M)$ such that $\mathcal{F}^M_\lambda(f)\in \mathcal{A}_{\lambda}$ for every $\lambda \in \sigma(\Delta_M)$ and we also define $\cinfty_{\mathcal{A}}(M) \dfn \cinfty(M)\cap \D'_\mathcal{A}(M)$.

When $M$ is a product $T\times G$ and $\mathcal A$ is a spectral cluster on $G$ we use the direct sum decomposition $E_\lambda^G = \mathcal{A}_\lambda \oplus \mathcal{A}_\lambda^\bot$ for each $\lambda \in \sigma(\Delta_G)$ to obtain 
\begin{align*}
\cinfty(T; E_\lambda^G) = \cinfty(T) \otimes E_\lambda^G = \left( \cinfty(T) \otimes \mathcal{A}_\lambda \right) \oplus  \left( \cinfty(T) \otimes \mathcal{A}_\lambda^\bot \right) \dfn \cinfty(T; \mathcal{A}_\lambda) \oplus \cinfty(T; \mathcal{A}_\lambda^\bot)
\end{align*}
in which the summands are orthogonal in $L^2(T \times G)$ thanks to Fubini's Theorem. We then define
\begin{align*}
  \cinfty_{\mathcal{A}}(T \times G) &\dfn \left \{f \in \cinfty(T \times G) \st \mathcal{F}_\lambda^G(f) \in \cinfty(T; \mathcal{A}_\lambda), \ \forall \lambda \in \sigma(\Delta_G) \right \}.
\end{align*}
It is easy to check that $\cinfty_{\mathcal{A}}(T\times G)$ is a closed subspace of $\cinfty(T\times G)$. In fact, if for each $\lambda \in \sigma(\Delta_G)$ we fix $\{\chi_1^\lambda, \ldots, \chi_{c_\lambda}^\lambda\}$ an orthonormal basis for $\mathcal{A}_\lambda$ and $\{\varphi_{c_\lambda+1}^\lambda,\ldots, \varphi_{d_\lambda^G}^\lambda\}$ an orthonormal basis for $\mathcal{A}_\lambda^\bot$ (where $c_\lambda \dfn \dim \mathcal A_\lambda$) then for any $f \in \cinfty(T\times G)$ we may write
\begin{align*}
  \mathcal F_\lambda^G(f) &= \sum_{i=1}^{c_\lambda} \langle f(t,\cdot), \chi_i^\lambda \rangle_{L^2(G)} \otimes \chi_i^\lambda + \sum_{i=c_\lambda+1}^{d_\lambda^G} \langle f(t, \cdot), \varphi_i^\lambda \rangle_{L^2(G)} \otimes \varphi_i^\lambda
\end{align*}
hence $f \in \cinfty_{\mathcal{A}}(T \times G)$ if and only if $\langle f(t, \cdot), \varphi_i^\lambda \rangle_{L^2(G)} = 0$ for every $i \in \{c_\lambda+1,\ldots,d_\lambda^G\}$ and $t \in T$. Suppose that $\{f_\nu \}_{\nu \in \N} \sset \cinfty_{\mathcal{A}}(T \times G)$ is a sequence that converges, in $\cinfty(T\times G)$, to some function $f \in \cinfty(T\times G)$: by the Dominated Convergence Theorem we have, for every $i \in \{c_\lambda+1,\ldots,d_\lambda^G\}$, that
\begin{align*}
  \int_G f(t,x) \overline{\varphi_i^\lambda(x)} \dd V_G(x) = \lim_{\nu \to \infty} \int_G f_\nu(t,x) \overline{\varphi_i^\lambda(x)} \dd V_G(x) = 0, \quad \forall t \in T,
\end{align*}
thus proving our claim. By the same token 
\begin{align*}
  \D'(T; E_\lambda^G) = \left( \D'(T) \otimes \mathcal{A}_\lambda \right) \oplus  \left( \D'(T) \otimes \mathcal{A}_\lambda^\bot \right) \dfn \D'(T; \mathcal{A}_\lambda) \oplus \D'(T; \mathcal{A}_\lambda^\bot)
\end{align*}
hence giving way to the definition
\begin{align*}
  \D'_{\mathcal{A}}(T \times G) &\dfn \{f \in \D'(T \times G) \st \mathcal{F}_\lambda^G(f) \in \D'(T; \mathcal{A}_\lambda), \ \forall \lambda \in \sigma(\Delta_G) \}
\end{align*}
which is a closed subspace of $\D'(T \times G)$. The proof of this fact is similar to the one done above: if $f \in \D'(T\times G)$ we may write
\begin{align*}
  \mathcal F_\lambda^G(f) &= \sum_{i=1}^{c_\lambda} \mathcal F_\lambda^G(f)_i \otimes \chi_i^\lambda + \sum_{i=c_\lambda+1}^{d_\lambda^G} \mathcal F_\lambda^G(f)_i \otimes \varphi_i^\lambda,
\end{align*}
so $f \in \D'_{\mathcal{A}}(T \times G)$ if and only if  $\mathcal F_\lambda^G(f)_i=0$ for every $i \in \{c_\lambda+1,\ldots,d_\lambda^G\}$. Suppose that $\{f_\nu \}_{\nu \in\N} \subset \D'_{\mathcal{A}}(T \times G)$ and that this sequence converges to $f \in \D'(T \times G)$  (where the convergence occurs in this space). For every $\psi \in \cinfty(T)$ and $i \in \{c_\lambda+1,\ldots,d_\lambda^G\}$ we have
\begin{align*}
  \langle  \mathcal F_\lambda^G(f)_i,\psi \rangle = \langle f,\psi \otimes \overline{\varphi_i^\lambda} \rangle = \lim_{\nu \to \infty} \langle f_\nu , \psi \otimes \overline{\varphi_i^\lambda} \rangle = \lim_{\nu \to \infty} \langle \mathcal F_\lambda^G(f_\nu)_i, \psi \rangle = 0.
\end{align*}

By using the natural projection
\begin{align*}
  \pi_{\mathcal{A}, \lambda}: \D'(T; E_\lambda^G) \longrightarrow \D'(T; \mathcal{A}_\lambda)
\end{align*}
we can define the Fourier projection on $\mathcal{A}_{\lambda}$ by 
\begin{align*}
  \mathcal{F}^{G, \mathcal{A}}_{\lambda} \dfn \pi_{\mathcal{A}, \lambda} \circ \mathcal{F}^G_{\lambda} :\D'(T \times G) \longrightarrow \D'(T; \mathcal{A}_\lambda)
\end{align*}
which is clearly written, in terms of the adapted orthonormal bases above, as
\begin{align}
  \mathcal{F}^{G, \mathcal{A}}_{\lambda}(f) &= \sum_{i=1}^{c_\lambda} \mathcal F_\lambda^G(f)_i \otimes \chi_i^\lambda, \quad f \in \D'(T \times G). \label{eq:FourieronAlambda}
\end{align}

\begin{Prop}\label{lem:convwrtA} If $\mathcal A = (\mathcal A_\lambda)_{\lambda \in \sigma(\Delta_G)}$ is a spectral cluster on $G$ and $f \in \D'(T \times G)$ then
  \begin{align}
    \pi_{\mathcal{A}} (f) &\dfn \sum_{\alpha \in \sigma(\Delta_M)} \sum_{\mu + \lambda = \alpha} \sum_{j=1}^{d^T_\mu} \sum_{i=1}^{c_\lambda} \left\langle f, \overline{\psi^\mu_j\otimes \chi^\lambda_i}\right\rangle \psi^\mu_j\otimes \chi^\lambda_i \label{eq:piA}
  \end{align}
  belongs to $\D'(T \times G)$ and satisfies
  \begin{align*}
    \mathcal{F}_\lambda^G(\pi_{\mathcal{A}} (f)) &= \mathcal{F}^{G, \mathcal{A}}_{\lambda}(f), \quad \forall \lambda \in \sigma(\Delta_G).
  \end{align*}
  In particular, $\pi_{\mathcal{A}} (f) \in \D'_{\mathcal{A}}(T \times G)$. Moreover, if $f\in \cinfty(T \times G)$ then $\pi_{\mathcal{A}} (f) \in \cinfty_{\mathcal{A}}(T \times G)$.
\end{Prop}
\begin{proof} It follows essentially by comparing terms in the formal Fourier series~\eqref{eq:piA} and~\eqref{eq:fourier_in_bases}, the latter suitably adapted to $\mathcal{A}, \mathcal{A}^\bot$, while taking into account~\eqref{eq:FourieronAlambda}: one sees at once that~\eqref{eq:piA} is exactly the full ``total'' Fourier series of $f$ with many terms removed, hence questions of convergence are easily decided by means of Corollary~\ref{cor:charac_smoothness_fourier_proj}.
\end{proof}

Therefore, we have a well-defined linear map $\pi_{\mathcal{A}}: \D'(T \times G) \rarr \D'_{\mathcal{A}}(T \times G)$ 
and it easily follows that
\begin{align*}
  f &= \pi_{\mathcal{A}} (f) + \pi_{\mathcal{A}^\bot} (f), \quad \forall f \in \D'(T \times G)
\end{align*}
since $\mathcal{F}^{G}_{\lambda}(f) = \mathcal{F}^{G, \mathcal{A}}_{\lambda}(f) + \mathcal{F}^{G, \mathcal{A}^\bot}_{\lambda}(f)$ for every $\lambda \in \sigma(\Delta_G)$. Notice that $f \in \D'_{\mathcal{A}}(T \times G)$ if and only if $\pi_{\mathcal{A}} (f) = f$, or, equivalently, $\pi_{\mathcal{A}^\bot} (f) = 0$.
\subsection{Invariant clusters on $T\times G$} \label{sec:invariant_clusters}

In this brief section we discuss a notion of a spectral cluster $\mathcal A \dfn (\mathcal A_\lambda)_{\lambda \in \sigma(\Delta_G)}$ on $G$ being invariant under the action of a LPDO $P$ on $M=T\times G$, and how such a property can be used to study issues of regularity of $P$.

We say that $\mathcal{A}$ is \emph{invariant} under $P$ if $P\big(\D'_{\mathcal{A}}(T \times G)\big)\subset \D'_{\mathcal{A}}(T \times G)$ and $P\big(\D'_{\mathcal{A}^\bot}(T \times G)\big)\subset \D'_{\mathcal{A}^\bot}(T \times G)$. This is equivalent to say that $P \circ \pi_{\mathcal{A}}= \pi_{\mathcal{A}} \circ P$ and $P \circ \pi_{\mathcal{A}^\bot}= \pi_{\mathcal{A}^\bot} \circ P$. In fact, suppose first that $\mathcal{A}$ is invariant under $P$ and let us prove the first identity: if $u \in \D'(T\times G)$ then $u = \pi_{\mathcal A}(u)+\pi_{\mathcal A^\bot}(u)$ and by hypothesis we have that
\begin{align*}
  Pu &= \underbrace{P\pi_{\mathcal A}(u)}_{ \in \D'_{\mathcal A}(T\times G)} + \underbrace{P\pi_{\mathcal A^\bot}(u)}_{ \in \D'_{\mathcal A^\bot}(T\times G)}.
\end{align*}
If we apply $\pi_{\mathcal A}$ to both sides of the equality above we obtain
\begin{align*}
  \pi_{\mathcal A}(Pu) = \pi_{\mathcal A} (P\pi_{\mathcal A}(u)) + \pi_{\mathcal A} (P\pi_{\mathcal A^\bot}(u)) = \pi_{\mathcal A} (P\pi_{\mathcal A}(u)) =  P\pi_{\mathcal A}(u),
\end{align*}
so $\pi_{\mathcal A} \circ P = P \circ \pi_{\mathcal A}$ since $u \in \D'(T\times G)$ is arbitrary; the proof that $P \circ \pi_{\mathcal{A}^\bot}= \pi_{\mathcal{A}^\bot} \circ P$ is analogous. Conversely, if $\pi_{\mathcal A} \circ P = P \circ \pi_{\mathcal A}$ and if $u \in \D'_{\mathcal A}(T\times G)$ then
\begin{align*}
  Pu = P \pi_{\mathcal{A}} (u) = \pi_{\mathcal{A}} (Pu) \in \D'_{\mathcal A}(T\times G).
\end{align*}
Analogously, if $\pi_{\mathcal A^\bot} \circ P = P \circ \pi_{\mathcal A^\bot}$ then $P\big(\D'_{\mathcal{A}^\bot}(T \times G)\big)\subset \D'_{\mathcal{A}^\bot}(T \times G)$.

We introduce the following refined notion of regularity of $P$; standard global hypoellipticity of $P$ in $T \times G$ -- $\mathrm{(GH)}$ for short -- corresponds to the full cluster $(E_\lambda^G)_{\lambda \in \sigma(\Delta_G)}$ in the definition below.
\begin{Def} We say that $P$ is~$\mathrm{(GH)}_{\mathcal{A}}$ if
  \begin{align*}
    \forall u \in \D'_{\mathcal{A}}(T \times G), \ Pu \in \cinfty(T \times G) &\Longrightarrow u \in \cinfty_{\mathcal{A}}(T \times G).
  \end{align*}
\end{Def}

\begin{Prop} \label{prop:GH-cluster-inv} Suppose that $\mathcal{A}$ is invariant under $P$. Then $P$ is~$\mathrm{(GH)}$ in $T \times G$ if and only if $P$ is both~$\mathrm{(GH)}_{\mathcal{A}}$ and~$\mathrm{(GH)}_{\mathcal{A}^\bot}$.
  \begin{proof} It is clear that if $P$ is~$\mathrm{(GH)}$ then $P$ is~$\mathrm{(GH)}_{\mathcal{A}}$ and~$\mathrm{(GH)}_{\mathcal{A}^\bot}$ without any additional assumption. For the converse, suppose that $P$ is~$\mathrm{(GH)}_{\mathcal{A}}$ and~$\mathrm{(GH)}_{\mathcal{A}^\bot}$ and let $u \in \D'(T\times G)$ be such that $Pu \in \cinfty(T\times G)$. Since $\mathcal A$ is invariant under $P$ we have that
    \begin{align*}
      P\pi_{\mathcal A}(u) = \pi_{\mathcal A} (P u) \in \cinfty (T\times G) \Longrightarrow \pi_{\mathcal A}(u) \in \cinfty(T\times G).
    \end{align*}
    Analogously, $\pi_{\mathcal A^\bot}(u) \in \cinfty(T\times G)$ and then $u = \pi_{\mathcal A}(u) +\pi_{\mathcal A^\bot} (u) \in \cinfty(T\times G)$. 
  \end{proof}
\end{Prop}
\section{Global solvability: general results} \label{sec:gs-gen}

We start this section by discussing a few notions of global solvability of operators in a general setting: given a compact, connected, oriented Riemannian manifold $M$, let us consider $P: \D'(M) \rarr \D'(M)$ a continuous linear operator that maps $\cinfty(M)$ into itself; consequently, $P: \cinfty(M) \rarr \cinfty(M)$ is also continuous by the Closed Graph Theorem. Its transpose $\PT$ enjoys the same property by reflexivity.

Let us introduce the notions of solvability that will be considered in this work.
\begin{Def} \label{def:solvability_compact} We say that $P$ is:
  \begin{enumerate}
  \item \label{it:gsSmooth} \emph{globally solvable in $\cinfty$} if for every $f \in \cinfty(M)$ satisfying
    \begin{align}
      \langle v, f \rangle &= 0 \ \text{for every $v \in \D'(M)$ such that $\PT v = 0$}  \label{eq:cinfty_cc}
    \end{align}
    there exists $u \in \cinfty(M)$ such that $Pu = f$;
  \item \label{it:gsDprime} \emph{globally solvable in $\D'$} if for every $f \in \D'(M)$ satisfying 
    \begin{align}
      \langle f, v \rangle &= 0 \ \text{for every $v \in \cinfty(M)$ such that $\PT v = 0$}\label{eq:distr_cc}
    \end{align}
    there exists $u \in \D'(M)$ such that $Pu = f$.
    \item \emph{weakly globally solvable} if for every $f \in \cinfty(M)$ satisfying~\eqref{eq:distr_cc} there exists $u \in \D'(M)$ such that $Pu = f$.
  \end{enumerate}
\end{Def}

It is easy to check the necessity of the compatibility conditions~\eqref{eq:cinfty_cc} and~\eqref{eq:distr_cc} for the existence of solutions.
These definitions are not standard as in many works ``global solvability'' refers to what we call weak global solvability here. The reason behind this terminology is that the latter condition holds automatically if~\eqref{it:gsDprime} is true, yet its relationship with~\eqref{it:gsSmooth} is more delicate since for a smooth $f$ the compatibility conditions~\eqref{eq:cinfty_cc} certainly imply~\eqref{eq:distr_cc} but the converse is not true in general. We explore this in Proposition~\ref{prop:gs_agh} under additional assumptions.

It is common knowledge in this business that when dealing with weak global solvability one may expect to find a necessary \emph{a priori} inequality for it, whose formulation requires us to consider the space
\begin{align}
  E &\dfn \left \{ f\in \cinfty(M) \st \int_{M} f v \ \dd V_M = 0, \ \forall v \in \cinfty(M), \ \PT v = 0 \right\} \label{eq:E}
\end{align}
of all $f \in \cinfty(M)$ satisfying the compatibility condition~\eqref{eq:distr_cc}. 
\begin{Lem} \label{lem:hormander_ineq} If $P$ is weakly globally solvable then there are $k_1, k_2 \in \Z_+$ and $C > 0$ such that
  \begin{align}\label{eq:equivtocontinuity}
    \left|\int_{M} f g \ \dd V_M \right| &\leq C \| f\|_{\sob^{k_1}(M)} \| \transp{P} g\|_{\sob^{k_2}(M)}
  \end{align}
  for every $f\in E$ and $g\in  \cinfty(M)$. 
  \begin{proof} We endow
    \begin{align*}
      F &\dfn \cinfty(M) / \ker \left\{ \transp{P}: \cinfty(M) \rarr \cinfty(M) \right\}
    \end{align*}
    with the metrizable topology defined by the seminorms
    \begin{align*}
      [g] \in F &\longmapsto \|\transp{P} g\|_{\sob^k(M)}, \quad k \in \Z_+,
    \end{align*}
    and $E$ with that inherited from $\cinfty(M)$. We consider the bilinear form $\Theta: E \times F \rarr \C$ defined by
    \begin{align*}
      \Theta(f, [g]) &\dfn \int_{M} f g \ \dd V_M.
    \end{align*}
    Note that if $g_1, g_2$ are representatives of the same class then $\transp{P}(g_1-g_2)=0$ hence it follows that $\Theta(f, [g_1])= \Theta(f, [g_2])$ for $f \in E$.
    
    Let us study the continuity of $\Theta$ in $E \times F$. Since the latter is the product of a Fr\'{e}chet space with a metrizable space, thanks to the Banach-Steinhaus Theorem it is enough to check that $\Theta$ is separately continuous to prove its continuity. It is clear that for every $[g] \in F$ we have that $E \ni f \mapsto \Theta(f, g)$ is continuous. Now for a given $f\in E$ the weak global solvability of $P$ yields $u \in \D'(M)$ such that $P u= f$. This implies that
    \begin{align*}
      \int_{M} f g \ \dd V_M &= \langle u, \transp{P} g \rangle, \quad \forall g \in \cinfty(M)
    \end{align*}
    By continuity of $u: \cinfty(M) \rarr \C$ it follows that there exist $C' > 0$ and $k \in \Z_+$ such that
    \begin{align*}
      |\Theta(f, [g])|\leq C' \|\transp{P} g\|_{\sob^k(M)}, \quad \forall  [g] \in F
    \end{align*}
    and then notice that $\Theta$ being continuous is equivalent to~\eqref{eq:equivtocontinuity}.
  \end{proof}
\end{Lem}


We will now recall an application of the Homomorphism Theorem for Fr{\'e}chet-Montel spaces. If we define
\begin{align*}
  A_1 \dfn P: \cinfty(M) \longrightarrow \cinfty(M), &\quad A_2 \dfn P: \D'(M) \longrightarrow \D'(M)
\end{align*}
then
\begin{align*}
  \text{$f \in \cinfty(M)$ satisfies~\eqref{eq:cinfty_cc}} &\Longleftrightarrow f \in \ker(\transp{A}_1)^\ort \\
  \text{$f \in \D'(M)$ satisfies~\eqref{eq:distr_cc}} &\Longleftrightarrow f \in \ker(\transp{A}_2)^\ort
\end{align*}
hence
\begin{align*}
  \text{$P$ is globally solvable in $\cinfty$} \Longleftrightarrow \ran(A_1) = \ker(\transp{A}_1)^\ort \Longleftrightarrow \text{$A_1$ has closed range}, \\
  \text{$P$ is globally solvable in $\D'$} \Longleftrightarrow \ran(A_2) = \ker(\transp{A}_2)^\ort \Longleftrightarrow \text{$A_2$ has closed range}.
\end{align*}
Our next proposition is then an immediate consequence of~\cite[p.~18]{kothe_tvs2}, keeping in mind that $\cinfty(M)$ is a Montel space.
\begin{Prop} \label{cor:hom_frechet_glob_solv} The following properties are equivalent.
  \begin{enumerate}
  \item $P: \cinfty(M) \rarr \cinfty(M)$ has closed range.
  \item $\PT: \D'(M) \rarr \D'(M)$ has closed range.
  \item $P$ is globally solvable in $\cinfty$.
  \item $\PT$ is globally solvable in $\D'$.
  \end{enumerate}  
In particular, global solvability in $\cinfty$ and in $\D'$ are equivalent if $P= \PT$.
\end{Prop}

Next we deduce a general sufficient condition for closedness of the range of $P: \cinfty(M) \rarr \cinfty(M)$ inspired by the notion of global hypoellipticity.

\begin{Def}\label{Def:AGH} We say that $P$ is:
  \begin{enumerate}
  \item \emph{globally hypoelliptic} in $M$ -- $\mathrm{(GH)}$ for short -- if
    \begin{align*}
      \forall u \in \D'(M), \ Pu \in \cinfty(M) &\Longrightarrow u \in \cinfty(M). 
    \end{align*}
  \item \emph{almost globally hypoelliptic} in $M$ -- $\mathrm{(AGH)}$ for short -- if
    \begin{align}
      \forall u \in \D'(M), \ Pu \in \cinfty(M) &\Longrightarrow \text{$\exists v \in \cinfty(M)$ such that $Pv = Pu$}. \label{eq:AGH}
    \end{align}
  \end{enumerate}
\end{Def}
It is clear that
\begin{align*}
  \text{$P$ is~$\mathrm{(GH)}$ in $M$} &\Longleftrightarrow \text{$P$ is~$\mathrm{(AGH)}$ in $M$ and $\ker \{ P: \D'(M) \longrightarrow \D'(M) \} \sset \cinfty(M)$}
\end{align*}
thus justifying the nomenclature. The reader should not get too excited about this analogy, however, as in extreme cases an operator may well be~$\mathrm{(AGH)}$ for trivial reasons while failing to be globally hypoelliptic by far (take the zero operator, for instance). Yet, the former property alone is enough to ensure global solvability.
\begin{Thm} \label{thm:general_agh_closedrange} If $P$ is~$\mathrm{(AGH)}$ in $M$ then $P: \cinfty(M) \rarr \cinfty(M)$ has closed range. In particular, if $P$ is~$\mathrm{(GH)}$ in $M$ then it is globally solvable in $\cinfty$.
\end{Thm}
\begin{proof} We denote
  \begin{align*}
    K \dfn \ker \{ P: \cinfty(M) \longrightarrow \cinfty(M) \}, &\quad K^\sharp \dfn \ker \{ P: \D'(M) \longrightarrow \D'(M) \}
  \end{align*}
  which are closed subspaces of their respective ambient spaces by continuity. In particular, $K_2 \dfn K^\sharp \cap L^2(M)$ is closed in $L^2(M)$ and therefore we may write $L^2(M) = K_2 \oplus K_2^\bot$ where $K_2^\bot$ denotes the $L^2$-orthogonal space to $K_2$.  Since $K$ is closed in $\cinfty(M)$ we have that $\cinfty(M)/K$ is a Fr{\'e}chet space: its topology is given by the seminorms $\{q_j\}_{j \in \Z_+}$ defined by:
  \begin{align*}
    q_j([u]) &\dfn \inf \left\{ \| v \|_{\sob^j(M)} \st v \in \cinfty(M), \ u - v \in K \right\}, \quad u \in \cinfty(M), \ j \in \Z_+.
  \end{align*}
\item The map
  \begin{align}
    [u] \in \cinfty(M)/K &\longmapsto Pu \in \cinfty(M) \label{eq:Preduced}
  \end{align}
  is then well-defined, continuous and injective, and we have a continuous injection $\cinfty(M)/K \hookrightarrow K_2^\bot$ given by $[u] \mapsto \pi(u)$ where $\pi: L^2(M) \rarr K_2^\bot$ denotes the orthogonal projection. Indeed, given $u, v \in \cinfty(M)$ we have that
  \begin{align*}
    [u] = [v] \Longleftrightarrow u - v \in K \sset K_2 \Longleftrightarrow \pi(u -v) = 0 \Longleftrightarrow \pi(u) = \pi(v).
  \end{align*}
  In particular, $\pi(u) = 0$ implies $[u] = [0]$, hence the injectivity. As for continuity, notice that for every $v \in \cinfty(M)$ such that $u - v \in K$ we have that
  \begin{align*}
    \| \pi(u)\|_{L^2(M)} = \| \pi(v)\|_{L^2(M)} \leq \| v \|_{L^2(M)} = \| v \|_{\sob^0(M)}
  \end{align*}
  so taking the infimum on the right-hand side yields $\| \pi(u)\|_{L^2(M)} \leq q_0([u])$ for every $u \in \cinfty(M)$.
  
  Now consider the continuous linear map between Fr{\'e}chet spaces
  \begin{align*}
    \TR{\gamma}{[u]}{\cinfty(M)/K}{(\pi(u), Pu)}{K_2^\bot \times \cinfty(M)}
  \end{align*}
  which is clearly injective. We claim that its range (which is simply the graph of~\eqref{eq:Preduced}, regarded as a subset of $K_2^\bot \times \cinfty(M)$) is closed. Indeed, let $\{u_\nu\}_{\nu \in \N} \sset \cinfty(M)$ and $(u, f) \in K_2^\bot \times \cinfty(M)$ be such that
  \begin{align*}
    \text{$\pi(u_\nu) \to u$ in $L^2(M)$}  \quad \text{and} \quad \text{ $Pu_\nu \to f$ in $\cinfty(M)$}.
  \end{align*}
  Since $u_\nu - \pi(u_\nu) \in K_2$ we have that $P u_\nu = P(\pi(u_\nu)) \to Pu$ in $\D'(M)$; yet, $Pu_\nu \to f$ in $\D'(M)$ hence $Pu = f \in \cinfty(M)$. By~\eqref{eq:AGH} there exists $v \in \cinfty(M)$ such that $Pv = Pu$, and thus $v - u \in K_2$ as both $u$ and $v$ belong to $L^2(M)$. Since $v = (v - u) + u$ we have by uniqueness in the direct sum decomposition that $\pi(v) = u$, that is, $u$ is precisely the image of $[v]$ under the injection $\cinfty(M)/K \hookrightarrow K_2^\bot$: in that sense, $(u,f)$ belongs to $Y \dfn \ran(\gamma)$.

  Therefore $Y$ is itself a Fr{\'e}chet space, and the inverse $\gamma^{-1}: Y \rarr \cinfty(M)/K$ is continuous by the Open Mapping Theorem. The latter property can then be quantified as follows: for each $j \in \Z_+$ there exist $C > 0$ and $k \in \Z_+$ such that
  \begin{align}
    q_j([u]) &\leq C \left( \| \pi(u) \|_{L^2(M)} + \| Pu \|_{\sob^k(M)} \right), \quad \forall u \in \cinfty(M). \label{eq:closed_graph_ineq}
  \end{align}
  We claim that there exists a constant $C' > 0$ such that
  \begin{align}
    q_j([u]) &\leq C' \| Pu \|_{\sob^k(M)}, \quad \forall u \in \cinfty(M).  \label{eq:closed_graph_ineq_plus}
  \end{align}
  There is no loss of generality in assuming that $j \geq 1$ since $q_0([u]) \leq q_1([u])$ for every $u \in \cinfty(M)$. So if there were no such constant $C'$ then for each $\nu \in \N$ there would be $u_\nu \in \cinfty(M)$ such that
  \begin{align*}
    q_j([u_\nu]) &> \nu \| Pu_\nu \|_{\sob^k(M)}
  \end{align*}
  and we may assume that $q_j([u_\nu]) = 1$. It follows that $Pu_\nu \to 0$ in $\sob^k(M)$, and also that for each $\nu \in \N$ there exists $v_\nu \in \cinfty(M)$ such that $u_\nu - v_\nu \in K$ and $\| v_\nu \|_{\sob^j(M)} \leq 2$ (by definition of $q_j$). Thus the sequence $\{ v_\nu \}_{\nu \in \N}$ is bounded in $\sob^j(M)$, and since the inclusion map $\sob^j(M) \hookrightarrow L^2(M)$ is compact we conclude that this sequence has a subsequence $\{ v_{\nu'} \}_{\nu' \in \N}$ which converges there, say, to some $v \in L^2(M)$. In particular, $Pu_{\nu'} = Pv_{\nu'} \to Pv$ in $\D'(M)$; but also $Pu_{\nu'} \to 0$ in $\D'(M)$, hence $v \in K_2$. By continuity of $\pi$ we have that $\pi(v_{\nu'}) \to \pi(v) = 0$ in $L^2(M)$ i.e.~$\| \pi(v_{\nu'}) \|_{L^2(M)} \to 0$, and since $u_\nu - v_\nu \in K$ we have that $\pi(u_\nu) = \pi(v_\nu)$ so $\| \pi(u_{\nu'}) \|_{L^2(M)} \to 0$ too. Plugging everything back into~\eqref{eq:closed_graph_ineq} yields
  \begin{align*}
    1 \leq C \left( \| \pi(u_{\nu'}) \|_{L^2(M)} + \| Pu_{\nu'} \|_{\sob^k(M)} \right) \to 0 \quad \text{as $\nu' \to \infty$}
  \end{align*}
  leading to a contradiction.

  We have so far proved that for each $j \in \Z_+$ there exist $C' > 0$ and $k \in \Z_+$ such that~\eqref{eq:closed_graph_ineq_plus} holds: this ultimately implies that the range of $P: \cinfty(M) \rarr \cinfty(M)$ is closed. Notice that the latter is exactly the range of~\eqref{eq:Preduced}, whose closedness in $\cinfty(M)$ we proceed to prove. Let $\{[u_\nu] \}_{\nu \in \N}$ be a sequence in $\cinfty(M)/K$ such that $\{ Pu_\nu \}_{\nu \in \N}$ converges to some $f$ in $\cinfty(M)$. Given $j \in \Z_+$ let $C' > 0$ and $k \in \Z_+$ be such that~\eqref{eq:closed_graph_ineq_plus} holds: since $Pu_\nu \to f$ in $\sob^k(M)$ it follows that $\{ Pu_\nu \}_{\nu \in \N}$ is a Cauchy sequence in $\sob^k(M)$, hence by~\eqref{eq:closed_graph_ineq_plus} $\{[u_\nu] \}_{\nu \in \N}$ is a Cauchy sequence w.r.t.~$q_j$, and since $j$ is arbitrary and $\{q_j\}_{j \in \Z_+}$ is a basis of seminorms for the topology of $\cinfty(M)/K$ we have that $\{[u_\nu] \}_{\nu \in \N}$ is a Cauchy sequence there. By completeness there exists $[u] \in \cinfty(M)/K$ such that $[u_\nu] \to [u]$ in $\cinfty(M)/K$, so by continuity of~\eqref{eq:Preduced} we have that $P u_\nu \to Pu$ in $\cinfty(M)$, hence $Pu = f$. Our proof is complete. 
\end{proof}

The converse of Theorem~\ref{thm:general_agh_closedrange} is false in general, as the next simple example shows. Yet, we have a useful partial converse, assuming an additional hypothesis on $\PT$ -- which will be satisfied by our classes of operators in the next sections -- that we prove in the sequence.

\begin{Exa} \label{exa:noconverseAGH} On $M \dfn S^1$ consider the smooth function defined by $a(x) \dfn 1 - e^{ix}$, which we regard as a zero order differential operator $P$. Since $a$ has a single zero, which is of order $1$, it is easily seen that the only distributions in $\ker \transp{P} = \ker P$ are multiples of the Dirac mass concentrated at $x = 0$, hence a smooth $f$ satisfies~\eqref{eq:cinfty_cc} if and only if $f(0) = 0$; on the other hand,~\eqref{eq:distr_cc} imposes no constraints.

  For $f \in \cinfty(S^1)$ vanishing at $x = 0$ by elementary calculus we may find a smooth $u$ such that $au = f$ i.e.~$P$ is globally solvable in $\cinfty$. Nevertheless, no smooth $v$ can satisfy $av = 1$ (as the right-hand side will always vanish at $0$) but it is classical that one can find a distribution $v$ with such a property, hence $P$ is not~$\mathrm{(AGH)}$ in $S^1$. Similar arguments ensure that the vector field $\vv{X} \dfn a(x) D_x$ on $S^1$ is globally solvable in $\cinfty$, but not~$\mathrm{(AGH)}$ there. 
\end{Exa}

\begin{Prop} \label{prop:gs_agh} If the kernel of $\transp{P}: \cinfty(M) \rarr \cinfty(M)$ is dense in the kernel of $\transp{P}: \D'(M) \rarr \D'(M)$ then for $f \in \cinfty(M)$ condition~\eqref{eq:distr_cc} implies condition~\eqref{eq:cinfty_cc}. If, moreover, $P: \cinfty(M) \rarr \cinfty(M)$ has closed range then $P$ is~$\mathrm{(AGH)}$ in $M$ and weakly globally solvable.
  \begin{proof} Take $f \in \cinfty(M)$ satisfying~\eqref{eq:distr_cc}. By hypothesis, any $v \in \D'(M)$ annihilated by $\transp{P}$ may be approximated in $\D'(M)$ by a net $\{ v_\alpha \}_{\alpha \in A}$ in $\ker \{\transp{P}: \cinfty(M) \rarr \cinfty(M)\}$, and we thus have
    \begin{align*}
      \langle v, f \rangle = \lim_\alpha \langle v_\alpha, f \rangle = \lim_\alpha \langle f, v_\alpha \rangle = 0
    \end{align*}
    hence~\eqref{eq:cinfty_cc} holds.
    
    Moreover, let $u \in \D'(M)$ be such that $f \dfn Pu \in \cinfty(M)$. Then $f$ satisfies~\eqref{eq:distr_cc}, hence also~\eqref{eq:cinfty_cc} by the previous arguments. If we further assume that $\ran \{P: \cinfty(M) \rarr \cinfty(M)\}$ is closed then $f$ belongs there, so there exists $v \in \cinfty(M)$ such that $Pv = f$, and therefore $P$~$\mathrm{(AGH)}$ in $M$. The conclusion that $P$ is also weakly globally solvable under such assumptions follows from a similar argument.
  \end{proof}
\end{Prop}

We can then restate our main conclusion as follows.
\begin{Cor} \label{cor:main_abstract} If the kernel of $\transp{P}: \cinfty(M) \rarr \cinfty(M)$ is dense in the kernel of $\transp{P}: \D'(M) \rarr \D'(M)$ then the following are equivalent:
  \begin{enumerate}
  \item $P$ is globally solvable in $\cinfty(M)$.
  \item $\transp{P}$ is globally solvable in $\D'(M)$.
  \item $P$ is~$\mathrm{(AGH)}$ in $M$.
  \end{enumerate}
\end{Cor}

\subsection{Almost globally hypoelliptic systems of vector fields} \label{sec:systems_AGH}

Given a family $\mathcal{L}$ of smooth, real vector fields on $M$ we define
\begin{align*}
  \ker \mathcal{L} &\dfn \{ u \in \D'(M) \st \vv{L}u = 0, \ \forall \vv{L} \in \mathcal{L} \}.
\end{align*}
\begin{Def} \label{def:agh_systems} We say that $\mathcal{L}$ is \emph{almost globally hypoelliptic} in $M$ -- $\mathrm{(AGH)}$ for short -- if for every $u \in \D'(M)$ we have
  \begin{align*}
    \vv{L} u \in \cinfty(M), \ \forall \vv{L} \in \mathcal{L} &\Longrightarrow \text{$\exists v \in \cinfty(M)$ such that $u - v \in \ker \mathcal{L}$}.
  \end{align*}
\end{Def}

\begin{Lem} \label{lem:LspanlieAGH} The following are equivalent:
  \begin{enumerate}
  \item \label{it:agh1} $\mathcal{L}$ is~$\mathrm{(AGH)}$ in $M$.
  \item \label{it:agh2} $\Span_\R \mathcal{L}$ is~$\mathrm{(AGH)}$ in $M$.
  \item \label{it:agh3} $\lie \mathcal{L}$, the Lie algebra generated by $\mathcal{L}$, is~$\mathrm{(AGH)}$ in $M$.
  \end{enumerate}
  \begin{proof} Given $\mathcal{L} \sset \mathcal{L}'$ two families of vector fields on $M$ such that $\ker \mathcal{L} = \ker \mathcal{L}'$ it is immediate to check that if $\mathcal{L}$ is~$\mathrm{(AGH)}$ in $M$ then so is $\mathcal{L}'$. Now, since $\mathcal{L} \sset \Span_\R \mathcal{L} \sset \lie \mathcal{L}$ we have that $\ker \lie \mathcal{L} \sset \ker \Span_\R \mathcal{L} \sset \ker \mathcal{L}$, but the latter inclusions are equalities since
    \begin{align*}
      \lie \mathcal{L} &= \Span_\R \bigcup_{\nu \in \N} \{[\vv{X}_1,[ \cdots [\vv{X}_{\nu - 1}, \vv{X}_\nu] \cdots ]] \st \vv{X}_j \in \mathcal{L}, \ 1 \leq j \leq \nu \}
    \end{align*}
    hence a distribution annihilated by every vector in $\mathcal{L}$ is clearly annihilated by those in $\lie \mathcal{L}$. We conclude that~$\eqref{it:agh1} \Rightarrow \eqref{it:agh2} \Rightarrow \eqref{it:agh3}$.
    
    If $u \in \D'(M)$ is such that $\vv{L} u \in \cinfty(M)$ for every $\vv{L} \in \mathcal{L}$ then also $\vv{L}' u \in \cinfty(M)$ for every $\vv{L}' \in \lie \mathcal{L}$, again due to the characterization of $\lie \mathcal{L}$ above; assuming~\eqref{it:agh3} we conclude that there exists $v \in \cinfty(M)$ such that $u - v \in \ker \lie \mathcal{L} = \ker \mathcal{L}$, so~\eqref{it:agh1} holds too.
  \end{proof}
\end{Lem}

Now we turn our attention to the case when $M$ is a compact Lie group $G$, which we will always endow with a Riemannian metric that is $\ad$-invariant, i.e., it is left-invariant and the inner product induced on the Lie algebra $\gr{g}$ turns the linear endomorphism $\vv{Y} \in \gr{g} \mapsto [\vv{X}, \vv{Y}] \in \gr{g}$ into a skew-symmetric map for every $\vv{X} \in \gr{g}$: the main advantage of this assumption is to make every element of $\gr{g}$ into a skew-symmetric operator w.r.t.~the $L^2(G)$ inner product; and these vector fields, moreover, commute with the associated Laplace-Beltrami operator $\Delta_G$. Metrics with this property always exist~\cite[Proposition~4.24]{knapp_lgbi}.

Given $\mathcal{L} \dfn \{\vv{L}_1, \ldots, \vv{L}_r\}$ a finite family of left-invariant vector fields on $G$, we proceed to show a few equivalent characterizations for its almost global hypoellipticity in $G$. First, it is clear that if we define the vector-valued differential operator $\mathrm{D}_{\mathcal{L}}: \cinfty(G) \rarr \cinfty(G)^r$ by
\begin{align*}
  \mathrm{D}_{\mathcal{L}} &\dfn ( \vv{L}_1, \ldots, \vv{L}_r )
\end{align*}
then $\mathcal{L}$ is~$\mathrm{(AGH)}$ in $G$ if and only if the same property holds for $\mathrm{D}_{\mathcal{L}}$ (with the obvious meaning). By results in~\cite{araujo19} the latter is equivalent to $\mathrm{D}_{\mathcal{L}}: \cinfty(G) \rarr \cinfty(G)^r$ having closed range, which is further characterized by the existence of constants $C, \rho > 0$ such that
\begin{align} \label{Eq:closed_range_systems_eq}
  \left( \sum_{j = 1}^r \| \vv{L}_j \phi \|_{L^2(G)}^2 \right)^{\frac{1}{2}} &\geq C(1 + \lambda)^{-\rho} \| \phi \|_{L^2(G)}, \quad \forall \phi \in E_\lambda^G \cap (\ker \mathcal{L})^\bot, \ \forall \lambda \in \sigma(\Delta_G). 
\end{align}

Now to $\mathcal{L}$ we may also associate a sublaplacian
\begin{align*}
  \Delta_{\mathcal{L}} &\dfn -\sum_{j = 1}^r \vv{L}_j^2
\end{align*}
which is a second-order scalar LPDO on $G$, and also left-invariant: by the very same results in~\cite{araujo19}, $\Delta_{\mathcal{L}}$ is~$\mathrm{(AGH)}$ in $G$ if and only if $\Delta_{\mathcal{L}}: \cinfty(G) \rarr \cinfty(G)$ has closed range, also equivalent to the inequality
\begin{align*}
  \| \Delta_{\mathcal{L}} \phi \|_{L^2(G)} &\geq C(1 + \lambda)^{-\rho} \| \phi \|_{L^2(G)}, \quad \forall \phi \in E_\lambda^G \cap (\ker \Delta_{\mathcal{L}})^\bot, \ \forall \lambda \in \sigma(\Delta_G)
\end{align*}
for some constants $C, \rho > 0$. Notice however that while it is certainly true that $\ker \mathcal{L} \sset \ker \Delta_{\mathcal{L}}$ we also have
\begin{align*}
  \langle \Delta_{\mathcal{L}} u, u \rangle_{L^2(G)} = - \sum_{j = 1}^r \langle \vv{L}_j^2 u, u \rangle_{L^2(G)} = \sum_{j = 1}^r \| \vv{L}_j u \|_{L^2(G)}^2, \quad u \in \cinfty(G),
\end{align*}
ensuring that $\ker \mathcal{L} \cap \cinfty(G) = \ker \Delta_{\mathcal{L}} \cap \cinfty(G)$ and leading to the validity of the reverse inclusion by a density argument.

Also, it is certain that if $\vv{L}_j u \in \cinfty(G)$ for every $j \in \{1, \ldots, r\}$ then $\Delta_{\mathcal{L}} u \in \cinfty(G)$, hence by the previous observation regarding the kernels we have that
\begin{align*}
  \text{$\Delta_{\mathcal{L}}$ is~$\mathrm{(AGH)}$ in $G$} &\Longrightarrow \text{$\mathcal{L}$ is~$\mathrm{(AGH)}$ in $G$}.
\end{align*}
Conversely, if $\mathcal{L}$ is~$\mathrm{(AGH)}$ in $G$ then there exist $C, \rho > 0$ such that, for every $\lambda \in \sigma(\Delta_G)$ and every $\phi \in E_\lambda^G \cap (\ker \mathcal{L})^\bot = E_\lambda^G \cap (\ker \Delta_{\mathcal{L}})^\bot$:
\begin{align*}
  \| \phi \|_{L^2(G)} \| \Delta_{\mathcal{L}} \phi \|_{L^2(G)} \geq \langle \Delta_{\mathcal{L}} \phi, \phi \rangle_{L^2(G)} = \sum_{j = 1}^r \| \vv{L}_j \phi \|_{L^2(G)}^2 \geq C^2 (1 + \lambda)^{-2 \rho} \| \phi \|_{L^2(G)}^2
\end{align*}
hence
\begin{align*}
  \| \Delta_{\mathcal{L}} \phi \|_{L^2(G)} &\geq C^2 (1 + \lambda)^{-2 \rho} \| \phi \|_{L^2(G)}
\end{align*}
that is, $\Delta_{\mathcal{L}}$ is~$\mathrm{(AGH)}$ in $G$: these considerations are summarized in the next statement.
\begin{Prop} \label{Prop:system_AGH-solv} For $\mathcal{L} \dfn \{\vv{L}_1, \ldots, \vv{L}_r\} \sset \gr{g}$ the following properties are equivalent:
  \begin{enumerate}
  \item $\mathcal{L}$ is~$\mathrm{(AGH)}$ in $G$.
  \item $\mathrm{D}_{\mathcal{L}}: \cinfty(G) \rarr \cinfty(G)^r$ has closed range.
  \item $\mathrm{D}_{\mathcal{L}}$ is~$\mathrm{(AGH)}$ in $G$.
  \item $\Delta_{\mathcal{L}}: \cinfty(G) \rarr \cinfty(G)$ has closed range.
  \item $\Delta_{\mathcal{L}}$ is~$\mathrm{(AGH)}$ in $G$.
  \end{enumerate}
\end{Prop}

For an arbitrary $\mathcal{L} \sset \gr{g}$, if $\mathcal L_0 \dfn \{\vv{L}_1, \ldots, \vv{L}_r\}$ and $\mathcal L_1 \dfn \{\vv{L}_1', \ldots, \vv{L}_r'\}$ are two bases of $\Span_\R \mathcal{L} \sset \gr{g}$ it follows from Lemma~\ref{lem:LspanlieAGH} that $\mathcal L_0$ is~$\mathrm{(AGH)}$ in $G$ if and only if the same property holds for $\mathcal L_1$. Now, by Proposition~\ref{Prop:system_AGH-solv}, $\mathcal L_i$ is~$\mathrm{(AGH)}$ in $G$ if and only if the associated operator $\mathrm{D}_{\mathcal{L}_i}: \cinfty(G) \rarr \cinfty(G)^r$ has closed range. This remark motivates the following:
\begin{Def} \label{Def:arbitrary-system-solv} Let $\mathcal L$ be an arbitrary system of left-invariant vector fields in $G$. We say that $\mathcal L$ is \emph{globally solvable in $\cinfty$} if there exists a basis $\mathcal L_0 \dfn \{\vv{L}_1, \ldots, \vv{L}_r\}$ of $\Span_\R \mathcal L$ such that the associated operator $\mathrm D_{\mathcal L_0}:\cinfty(G) \longrightarrow \cinfty(G)^r$ has closed range.
\end{Def}

While it is well known that tori may carry real, left-invariant vector fields which are globally hypoelliptic~\cite{gw72}, the circumstances are not so favorable for non-Abelian compact Lie groups $G$~\cite{gw73b}; actually, no commutative subalgebra of $\gr{g}$ enjoys that property~\cite[Corollary~8.9]{afr20}. As for \emph{almost} global hypoellipticity, the situation is radically different as the next simple example shows; the reader is referred to e.g.~\cite[Chapter~11]{rt_psdoas} for omitted details and computations.

\begin{Exa}[An $\mathrm{(AGH)}$ vector field on $\SU(2)$] \label{exa:su2} On the special unitary group $G = \SU(2)$, the real left-invariant vector field $\vv{X}$ corresponding to the matrix
  \begin{align*}
    \frac{1}{2} \left(
    \begin{array}{c c}
      i & 0 \\
      0 & -i
    \end{array}
    \right) &\in \gr{su}(2)
  \end{align*}
  (hence $\vv{X} = -i \del_0$, where $\del_0$ is the so-called \emph{neutral operator} of $\SU(2)$) enjoys the following property. For each $\lambda \in \sigma(\Delta_G)$ there exists an orthonormal basis of $E_\lambda^G$ (w.r.t.~a suitable $\ad$-invariant metric)
  \begin{align*}
    \chi^\lambda_1, \ldots, \chi^\lambda_{c_\lambda}, \varphi^\lambda_{c_\lambda + 1}, \ldots, \varphi^\lambda_{d_\lambda^G}
  \end{align*}
  (with $c_\lambda \geq 1$ for $\lambda \neq 0$) formed by eigenvectors of $\vv{X}$~\cite[Proposition~11.9.2]{rt_psdoas}:
  \begin{align*}
    \vv{X} \chi_j^\lambda &= 0, \quad j \in \{1, \ldots, c_\lambda\} \\
    \vv{X} \varphi_j^\lambda &= \gamma_j^\lambda \varphi_j^\lambda, \quad j \in \{c_\lambda + 1, \ldots, d_\lambda^G\}
  \end{align*}
  with $|\gamma_j^\lambda| \geq 1/2$ for every $j \in \{c_\lambda + 1, \ldots, d_\lambda^G\}$. In particular, a general $\phi \in E_\lambda^G \cap (\ker \vv{X})^\bot$ is written as
  \begin{align*}
    \phi &= \sum_{j = c_\lambda + 1}^{d_\lambda^G} \langle \phi, \varphi_j^\lambda \rangle_{L^2(G)}  \varphi_j^\lambda
  \end{align*}
  and thus
  \begin{align*}
    \vv{X} \phi = \sum_{j = c_\lambda + 1}^{d_\lambda^G} \langle \phi, \varphi_j^\lambda \rangle_{L^2(G)} \gamma_j^\lambda \varphi_j^\lambda &\Longrightarrow \| \vv{X} \phi \|_{L^2(G)}^2 = \sum_{j = c_\lambda + 1}^{d_\lambda^G} |\langle \phi, \varphi_j^\lambda \rangle_{L^2(G)}|^2 |\gamma_j^\lambda|^2 \geq \frac{1}{4} \| \phi\|_{L^2(G)}^2
  \end{align*}
  from which it follows that an inequality like~\eqref{Eq:closed_range_systems_eq} holds for $\vv{X}$, which is then~$\mathrm{(AGH)}$ in $G$. Notice that $\chi^\lambda_1, \ldots, \chi^\lambda_{c_\lambda} \in \ker \vv{X}$ and $c_\lambda \geq 1$ for every $\lambda \in \sigma(\Delta_G)$, hence the kernel of $\vv{X}$ is infinite dimensional: in particular, $\vv{X}$ is not~$\mathrm{(GH)}$ in $G$.
\end{Exa}

\section{Sums of squares of tube type} \label{Sec:tube-type}

Let $T, G$ be two compact manifolds as in Section~\ref{sec:partial_exp}. Here, however, we shall assume that $G$ is a Lie group (whose Lie algebra we still denote by $\gr{g}$), and while we impose no extra condition on the Riemannian metric on $T$ the one on $G$ will be assumed $\ad$-invariant. On time, we will assume by simplicity that both $T$ and $G$ have total measure equal to $1$.

As in~\cite{afr20} we will consider differential operators on $M = T \times G$ of the form
\begin{align}
  P &\dfn \Delta_T - \sum_{\ell = 1}^N \Big( \sum_{j = 1}^m a_{\ell j}(t) \vv{X}_j + \vv{W}_\ell \Big)^2 \label{eq:Pdef}
\end{align}
where $\vv{X}_1, \ldots, \vv{X}_m \in \gr{g}$ form a basis of left-invariant vector fields in $G$, $a_{\ell j} \in \cinfty(T; \R)$ and $\vv{W}_\ell$ are skew-symmetric real vector fields in $T$. In order to shorten the notation let $\gr{a}_1, \ldots, \gr{a}_N: T \rarr \gr{g}$ be smooth applications defined by
\begin{align*}
  \mathfrak{a}_\ell(t) &\dfn \sum_{j = 1}^m a_{\ell j}(t) \vv{X}_j, \quad t \in T, 
\end{align*}
which we will denote by $\gr{a}_\ell(t, \vv{X})$ when we want to stress their interpretation as differential operators on $T \times G$. We will denote by $\mathcal{L}$ the system of vector fields on $G$ defined as follows:
\begin{align}
  \mathcal{L} &\dfn \bigcup_{\ell = 1}^N \ran \gr{a}_\ell \sset \gr{g} \label{eq:sys_fromaell}
\end{align}
where a left-invariant vector field $\vv{L}$ belongs to $\ran \gr{a}_\ell$ if and only if there exists $t \in T$ such that $\vv{L} = \gr{a}_\ell(t)$. For each $\ell \in \{1, \ldots, N\}$ we also set
\begin{align}
  \mathcal{L}_\ell &\dfn \Span_\R \ran \mathfrak{a}_\ell \sset \gr{g}. \label{eq:Lell}
\end{align}
Recall that by restriction to functions on $T \times G$ that do not depend on the second variable, $P$ induces an elliptic operator in $T$ (see~\cite[Section 4]{afr20})
\begin{align}\label{eq:Ptil}
  \tilde{P} &\dfn \Delta_T - \sum_{\ell = 1}^N  \vv{W}_\ell^2
\end{align}
and thanks to~\cite[Corollary~4.2]{afr20} we know that if $u \in \D'(T \times G)$ is such that $Pu\in \cinfty(T\times G)$ the ellipticity of $\tilde{P}$ implies that $\mathcal{F}_\lambda^G(u)\in \cinfty(T; E^G_{\lambda})$ for every $\lambda\in \sigma(\Delta_G)$.

The first step to characterize the global solvability of $P$ is to describe $\ker \PT$. In our case, the operator $P$ is self-adjoint therefore $\ker \PT= \ker P$. For the next result we also recall that since $P$ commutes with $\Delta_G$ it also commutes with the partial Fourier projections $\mathcal F_\lambda^G$~\cite[Proposition~2.11]{afr20}. Here and below $\phi_1^{\lambda}, \ldots, \phi_{d_\lambda^G}^{\lambda}$ will denote any orthonormal basis of $E^G_\lambda$.
\begin{Prop}\label{Pro:kernel} For $P$ as in~\eqref{eq:Pdef}, a distribution $u \in \D'(T \times G)$ belongs to $\ker P$ if and only if $u$ does not depend on $T$ and $u \in \ker \mathfrak{a}_\ell(t)$ for every $t \in T$ and $\ell \in \{1,\ldots,N\}$.
  \begin{proof} It is immediate from its definition~\eqref{eq:Pdef} that $P$ annihilates any distribution that does not depend on $T$ and belongs to $\ker \mathfrak{a}_\ell(t)$ for every $t \in T$ and $\ell \in \{1,\ldots,N\}$. Conversely, if $Pu = 0$ we have that $\mathcal{F}_\lambda^G (u) \in \cinfty(T; E^G_{\lambda})$ for every $\lambda\in \sigma(\Delta_G)$, and it follows from~\cite[Lemma~3.1]{afr20} that
    \begin{align} \label{Pro:kernel-eq1}
      \langle P \mathcal F_\lambda^G (u), \mathcal F_\lambda^G (u) \rangle_{L^2(T\times G)} &= \langle \Delta_T \mathcal F_\lambda^G (u),\mathcal F_\lambda^G (u) \rangle_{L^2(T\times G)} + \sum_{\ell=1}^N \| \vv{Y}_\ell \mathcal F_\lambda^G (u) \|^2_{L^2(T\times G)}
    \end{align}
    equals zero, where
    \begin{align} \label{Eq:def_y_ell}
      \vv{Y}_\ell &\dfn \mathfrak a_\ell(t,\vv{X}) + \vv{W}_\ell, \quad \ell \in \{1,\ldots,N\}.
    \end{align}
    Notice that all the terms in the right-hand side of~\eqref{Pro:kernel-eq1} are non-negative, hence equal to zero. If we write
    \begin{align*}
      \mathcal F_\lambda^G (u) = \sum_{i=1}^{d^G_\lambda} \mathcal F_\lambda^G (u)_i \otimes \phi_i^\lambda, \quad \mathcal F_\lambda^G (u)_i \in \cinfty(T),
    \end{align*}
    then
    \begin{align} \label{eq:equality-energy-part}
      \langle \Delta_T \mathcal F_\lambda^G (u),\mathcal F_\lambda^G (u) \rangle_{L^2(T\times G)} &= \sum_{i=1}^{d_\lambda^G} \|\dd_T \mathcal F_\lambda^G (u)_i\|^2_{L^2(T)} 
    \end{align}
    is also zero. In fact,
    \begin{align*}
      \langle \Delta_T \mathcal F_\lambda^G (u), \mathcal F_\lambda^G (u) \rangle_{L^2(T\times G)} &= \left \langle \Delta_T \sum_{i=1}^{d^G_\lambda} \mathcal F_\lambda^G (u)_i \otimes \phi_i^\lambda, \sum_{j=1}^{d^G_\lambda} \mathcal F_\lambda^G (u)_j \otimes \phi_j^\lambda \right\rangle_{L^2(T\times G)}\\
      &= \left \langle  \sum_{i=1}^{d^G_\lambda} [\Delta_T \mathcal F_\lambda^G (u)_i] \otimes \phi_i^\lambda, \sum_{j=1}^{d^G_\lambda} \mathcal F_\lambda^G (u)_j \otimes \phi_j^\lambda \right\rangle_{L^2(T\times G)}\\
      &= \sum_{i=1}^{d_\lambda^G} \langle  \Delta_T \mathcal F_\lambda^G (u)_i,  \mathcal F_\lambda^G (u)_i \rangle_{L^2(T)}\\
      &= \sum_{i=1}^{d_\lambda^G} \langle  \dd_T \mathcal F_\lambda^G (u)_i, \dd_T \mathcal F_\lambda^G (u)_i \rangle_{L^2(T)}\\
      &= \sum_{i=1}^{d_\lambda^G} \|\dd_T \mathcal F_\lambda^G (u)_i \|^2_{L^2(T)}.
    \end{align*}
    Since $T$ is connected it follows that $\mathcal F_\lambda^G (u)_i$ is constant for every $i \in \{1,\ldots,d_\lambda^G\}$, so $\Delta_T \mathcal F_\lambda^G(u) = 0$ and, since $\Delta_T$ commutes with $\Delta_G$,
    \begin{align*}
      \mathcal{F}_\lambda^G (\Delta_T u) = \Delta_T \mathcal{F}_\lambda^G (u) = 0, \quad \forall \lambda \in \sigma(\Delta_G).
    \end{align*}
    It follows from relationship~\eqref{eq:total_from_partials} applied to $f \dfn \Delta_T u$ that $\mathcal{F}_\alpha^M (\Delta_T u) = 0$ for every $\alpha \in \sigma(\Delta_M)$ hence $\Delta_T u = 0$, that is, $u$ does not depend on $T$ -- see Lemma~\ref{Lem:dist-const-T}. In particular, $\vv{W}_\ell u = 0$ for every $\ell \in \{1, \ldots, N\}$ so it easily follows from~\eqref{Pro:kernel-eq1} that
    \begin{align*}
      0 = \sum_{\ell=1}^N \| \vv{Y}_\ell \mathcal F_\lambda^G (u) \|^2_{L^2(T\times G)} = \sum_{\ell=1}^N \| \mathfrak a_\ell(t, \vv{X}) \mathcal F_\lambda^G (u) \|^2_{L^2(T\times G)}
    \end{align*}
    meaning that
    \begin{align*}
      \sum_{j = 1}^m a_{\ell j}(t) \left( \vv{X}_j \mathcal F_\lambda^G (u) \right)(x) &= 0, \quad \forall (t,x) \in T \times G
    \end{align*}
    i.e.~$\mathcal{F}_\lambda^G (u) \in E_\lambda^G$ is such that $\gr{a}_\ell(t) \mathcal{F}_\lambda^G (u) = 0$ for every $\lambda \in \sigma(\Delta_G)$ and every $t \in T$, from where our conclusion follows.
  \end{proof}
\end{Prop}
\subsection{The cluster associated with $P$} \label{sec:cluster_P}

Motivated by Proposition~\ref{Pro:kernel} we start this section studying the set of distributions on $G$ which are annihilated by $\mathfrak{a}_{\ell}(t) \in \gr{g}$ for every $t \in T$, where $\ell \in \{1, \ldots, N\}$ is fixed: we may rewrite our conclusion as
\begin{align*}
  \ker P &= \{ 1_T \otimes v \st v \in \D'(G), \ v \in \ker \gr{a}_\ell(t) \ \forall t \in T, \ \ell \in \{1, \ldots, N\} \}.
\end{align*}
From here on we denote $m^{\ell}\dfn \dim \mathcal{L}_\ell$ and fix $\vv{L}_1^\ell, \ldots, \vv{L}_{m^\ell}^\ell$ a basis of $\mathcal{L}_\ell$~\eqref{eq:Lell}: it allows us to write
\begin{align}
  \mathfrak{a}_\ell(t) &= \sum_{p=1}^{m^\ell} \alpha_{\ell p}(t) \vv{L}^\ell_p,\label{eq:novoa}
\end{align}
where $\alpha_{\ell 1},\ldots,\alpha_{\ell m^\ell}$ are $\R$-linearly independent elements of $\cinfty(T; \R)$. Making use of the notation introduced in Section~\ref{sec:systems_AGH} we have:

 

\begin{Lem} \label{Lem:kernel_a_ell} Given $\ell \in \{1,\ldots,N\}$ we have the following equality of sets:
  \begin{align*}
    \bigcap_{t \in T} \ker \gr{a}_\ell(t) = \ker ( \ran \gr{a}_\ell) = \ker \mathcal{L}_\ell = \bigcap_{p=1}^{m^\ell} \ker \vv{L}^\ell_p.
  \end{align*}
  \begin{proof} The first identity follows by definition; the second one, since every vector field in $\mathcal{L}_\ell$ is a linear combination of elements in $\ran \gr{a}_\ell$ while containing the latter; and the third one follows by similar arguments.
  \end{proof}
\end{Lem}

We will consider $\mathcal{A}= (\mathcal{A}_{\lambda})_{\lambda\in \sigma(\Delta_G)}$ the spectral cluster on $G$ given by
\begin{align}
  \mathcal{A}_{\lambda} &\dfn E^G_{\lambda} \cap \Big( \bigcap_{\ell=1}^{N}\ker ( \ran \gr{a}_\ell)\Big), \quad \lambda \in \sigma(\Delta_G). \label{spectralclusterqueusamos}
\end{align}
The main reason to introduce it is that now we may rewrite the conclusions of Proposition~\ref{Pro:kernel} and Lemma~\ref{Lem:kernel_a_ell} as
\begin{align*}
  \ker \{P: \cinfty(T \times G) \longrightarrow \cinfty(T \times G)\} &= \{ 1_T \otimes v \st v \in \cinfty_{\mathcal{A}}(G)\},\\
  \ker \{P: \D'(T \times G) \longrightarrow \D'(T \times G)\} &= \{ 1_T \otimes v \st v \in \D'_{\mathcal{A}}(G)\}.
\end{align*}
It follows that our operator $P$ satisfies the hypothesis of Proposition~\ref{prop:gs_agh} and Corollary~\ref{cor:main_abstract}:
\begin{Prop} \label{Prop:P_kernel_dense} The kernel of $P: \cinfty(T\times G) \rarr \cinfty(T \times G)$ is dense in the kernel of $P: \D'(T\times G) \rarr \D'(T\times G)$.
  \begin{proof} Any $u \in \D'(T \times G)$ such that $Pu = 0$ is of the form $u = 1_T \otimes v$ for some $v \in \D'_{\mathcal{A}}(G)$. The latter can be written as
    \begin{align*}
      v &= \sum_{\lambda \in \sigma(\Delta_G)} \mathcal{F}_\lambda^G (v), \quad \mathcal{F}_\lambda^G (v) \in \mathcal{A}_\lambda,
    \end{align*}
    with convergence in $\D'(G)$, hence 
    \begin{align*}
      v_\nu &\dfn \sum_{\substack{\lambda \in \sigma(\Delta_G)\\ \lambda \leq \nu}} \mathcal{F}_\lambda^G (v)
    \end{align*}
    defines a sequence $\{ v_\nu \}_{\nu \in \N} \sset \cinfty_{\mathcal{A}}(G)$ such that $v_\nu \to v$ in $\D'(G)$. Therefore $u_\nu \dfn 1_T \otimes v_\nu \in \cinfty(T \times G)$ is annihilated by $P$ and converges to $u$ in $\D'(T \times G)$: we proved the density of $ \ker\{ P: \cinfty(T \times G) \to \cinfty(T \times G)\}$ in  $ \ker \{ P: \D'(T \times G) \to \D'(T \times G)\}$.
  \end{proof}
\end{Prop}
Whenever it helps in the computations below we shall employ orthonormal bases adapted to $\mathcal{A}$: for each $\lambda\in \sigma(\Delta_G)$ we will denote by $\chi^\lambda_1, \ldots, \chi^\lambda_{c_\lambda}$ an orthonormal basis of $\mathcal{A}_{\lambda}$ and by $\varphi_{c_\lambda+1}^{\lambda},\ldots, \varphi_{d_\lambda^G}^{\lambda}$ an orthonormal basis of $\mathcal{A}_{\lambda}^\bot$. Then we have the following:

\begin{Prop} \label{Prop:invariance-cluster} The cluster $\mathcal{A}$ defined above is invariant under the action of $P$. 
  \begin{proof} Rewriting $P$ using~\eqref{eq:novoa}, one checks that its action on a tensor product is given by the formula
    \begin{multline} \label{eq:action_P_tensor}
      P(\psi \otimes \phi) = (\tilde{P}\psi) \otimes \phi + {} \\
      - \sum_{\ell = 1}^N \left( \sum_{p,p' = 1}^{m^\ell} (\alpha_{\ell p'} \alpha_{\ell p} \psi) \otimes (\vv{L}^\ell_{p'} \vv{L}^\ell_p \phi) + \sum_{p = 1}^{m^\ell} \left( (2 \alpha_{\ell p} \vv{W}_\ell + \vv{W}_\ell \alpha_{\ell p}) \psi \right) \otimes (\vv{L}^\ell_p \phi) \right)
    \end{multline}
    where $\tilde P$ was defined in~\eqref{eq:Ptil}. Let us prove that $\D'_{\mathcal A}(T\times G)$ is invariant by $P$: if $u \in \D'_{\mathcal A}(T\times G)$ then 
    \begin{align*}
      \mathcal F_\lambda^G (u) &= \sum_{i=1}^{c_\lambda} \mathcal F_\lambda^G (u)_i \otimes \chi^\lambda_i
    \end{align*}
    where $\mathcal F_\lambda^G (u)_i \in \D'(T)$ is uniquely determined. Using that $\chi^\lambda_i \in \ker \vv{L}_p^\ell$ for every $i \in \{1,\ldots, c_\lambda\}$, $p \in \{1, \ldots, m^\ell\}$ and $\ell \in \{1,\ldots, N\}$ thanks to Lemma~\ref{Lem:kernel_a_ell}, and making use of~\eqref{eq:action_P_tensor} as well as the fact that $P$ commutes with $\mathcal F_\lambda^G$~\cite[Proposition~2.11]{afr20} we have
    \begin{align*}
      \mathcal F_\lambda^G (P u) = P \mathcal F_\lambda^G (u) = \sum_{i=1}^{c_\lambda} P\left(\mathcal F_\lambda^G (u)_i\otimes \chi_i^\lambda\right) = \sum_{i=1}^{c_\lambda} \left(\tilde{P} [\mathcal F_\lambda^G (u)_i] \right)\otimes \chi_i^\lambda \in \D'(T;\mathcal A_\lambda).
    \end{align*}
    Thus we have proved that $P\left(\D'_{\mathcal A}(T\times G)\right) \subset \D'_{\mathcal A}(T\times G)$ (see Remark~\ref{rmk_P_tildeP_cluster} below).

    In order to prove the invariance of $\D'_{\mathcal A^\bot}(T\times G)$ we take $u$ there and as before we have to prove that
    \begin{align*}
      \mathcal F_\lambda^G (P u) &= P \mathcal F_\lambda^G (u) \in \D'(T;\mathcal A_\lambda^\bot), \quad \forall \lambda \in \sigma(\Delta_G).
    \end{align*}
    To do so, recall that every $\vv{X} \in \gr{g}$ is skew-symmetric w.r.t.~the $L^2(G)$ inner product hence given $\varphi \in \mathcal{A}_\lambda^\bot$ we have
    \begin{align*}
      \langle \vv{L}_p^\ell \varphi, \chi \rangle_{L^2(G)} = - \langle \varphi, \vv{L}_p^\ell \chi \rangle_{L^2(G)} = 0, \quad \forall \chi \in \mathcal{A}_\lambda
    \end{align*}
    i.e.~$\vv{L}_p^\ell \varphi \in \mathcal{A}_\lambda^\bot$ for every $p \in \{1, \ldots, m^\ell\}$ and $\ell \in \{1,\ldots, N\}$. Therefore, writing
    \begin{align*}
      \mathcal F_\lambda^G (u) &= \sum_{i=c_\lambda+1}^{d_\lambda^G} \mathcal F_\lambda^G (u)_i \otimes \varphi^{\lambda}_{i}
    \end{align*}
    we obtain from~\eqref{eq:action_P_tensor} the following expression for $P \mathcal F_\lambda^G (u)$:
    \begin{multline*}
      P \mathcal F_\lambda^G (u) = \sum_{i=c_\lambda+1}^{d_\lambda^G} \left (\tilde P[\mathcal F_\lambda^G (u)_i] \right) \otimes \varphi_i^\lambda + {}\\
      - \sum_{i=c_\lambda+1}^{d_\lambda^G} \sum_{\ell = 1}^N \left( \sum_{p, p' = 1}^{m^\ell} (\alpha_{\ell p'} \alpha_{\ell p} \mathcal F_\lambda^G (u)_i) \otimes (\vv{L}^\ell_{p'} \vv{L}^\ell_p \varphi_i^\lambda) + \sum_{p = 1}^{m^\ell} ((2 \alpha_{\ell p} \vv{W}_\ell + \vv{W}_\ell \alpha_{\ell p})\mathcal F_\lambda^G (u)_i) \otimes (\vv{L}^\ell_p \varphi_i^\lambda) \right)
    \end{multline*}
    which certainly belongs to $\D'(T; \mathcal{A}_\lambda^\bot) = \D'(T) \otimes \mathcal{A}_\lambda^\bot$ thanks to our previous argument.
  \end{proof}
\end{Prop}

\begin{Rem} \label{rmk_P_tildeP_cluster} Now we take advantage of the last proof to obtain an equality that will be important for us afterwards:
  \begin{align} \label{eq:P_tildeP_cluster}
    P u &= \tilde{P} u, \quad \forall u \in \D'_{\mathcal A}(T\times G).
  \end{align}
  Indeed, for every $\lambda \in \sigma(\Delta_G)$ we have that
  \begin{align*}
    \mathcal F_\lambda^G (P u) = \sum_{i=1}^{c_\lambda} \left(\tilde P [\mathcal F_\lambda^G (u)_i]\right) \otimes \chi_i^\lambda = \tilde P \mathcal F^G_\lambda (u) = \mathcal F^G_\lambda (\tilde P u).
  \end{align*}
\end{Rem}

\section{A characterization of the global solvability of $P$} \label{sec:charac-P}

Our main result concerns the differential operator $P$ defined by~\eqref{eq:Pdef} on $T \times G$. In order to proceed we shall make use of the following additional hypothesis:
\begin{align} \label{it:thm15_hyp1}
  \text{for each $\ell \in \{1, \ldots, N\}$, $\mathfrak{a}_\ell(t_1)$ and $\mathfrak{a}_\ell(t_2)$  commute as vector fields in $G$, for any $t_1, t_2 \in T$}.
\end{align}
Also, recall that by $\mathcal{L}$ we denote the system of vector fields defined in~\eqref{eq:sys_fromaell}. Here is our statement.
\begin{Thm} \label{Thm:main-P} Let $P$ be as in~\eqref{eq:Pdef} and assume property~\eqref{it:thm15_hyp1}. The following are equivalent:
    \begin{enumerate}
    \item \label{it:thmmain1} $P$ is globally solvable in $\cinfty$;
    \item \label{it:thmmain2} $P$ is globally solvable in $\D'$;
    \item \label{it:thmmain3} $P$ is weakly globally solvable;
    \item \label{it:thmmain4} $P$ is~$\mathrm{(AGH)}$ in $T \times G$;
    \item \label{it:thmmain5} $\mathcal{L}$ is~$\mathrm{(AGH)}$ in $G$;
    \item \label{it:thmmain6} $\mathcal{L}$ is globally solvable in $\cinfty$.   
  \end{enumerate} 
\end{Thm}

Note that Proposition~\ref{cor:hom_frechet_glob_solv} yields equivalence between~\eqref{it:thmmain1} and~\eqref{it:thmmain2}, which in turn imply~\eqref{it:thmmain3}. Thanks to Corollary~\ref{cor:main_abstract} and Proposition~\ref{Prop:P_kernel_dense}, we have that~\eqref{it:thmmain1} and~\eqref{it:thmmain4} are equivalent, while Proposition~\ref{Prop:system_AGH-solv} shows the equivalence between~\eqref{it:thmmain5} and~\eqref{it:thmmain6}: the only implications left to prove are $\eqref{it:thmmain3} \Rightarrow \eqref{it:thmmain5}$ and $\eqref{it:thmmain5} \Rightarrow \eqref{it:thmmain4}$, to which we turn our attention in the next couple of sections.


\subsection*{Necessity of $\mathcal{L}$ to be~$\mathrm{(AGH})$ for weak global solvability of $P$}

Recall that $\mathcal{A}$ stands for the spectral cluster given by~\eqref{spectralclusterqueusamos}.
\begin{Prop} If $P$ is weakly globally solvable then $\mathcal{L}$ is~$\mathrm{(AGH)}$ in $G$.
  \begin{proof} First notice that since
    \begin{align*}
      \left\{ \vv{L}_p^\ell \st p \in \{1, \ldots, m^\ell\}, \ \ell \in \{1, \ldots, N\} \right\} 
    \end{align*}
    generates $\Span_\R \mathcal{L}$, Lemma~\ref{lem:LspanlieAGH} says that it is enough to check that the former is~$\mathrm{(AGH)}$. We then proceed to check almost global hypoellipticity of this system, which is characterized -- see~\eqref{Eq:closed_range_systems_eq} and Lemma~\ref{Lem:kernel_a_ell} -- by the existence of constants $C, \rho > 0$ such that
    \begin{align}
      \bigg( \sum_{\ell = 1}^N \sum_{p = 1}^{m^\ell} \| \vv{L}_p^\ell \phi \|_{L^2(G)}^2 \bigg)^{\frac{1}{2}} &\geq C (1 + \lambda)^{-\rho} \| \phi \|_{L^2(G)}, \quad \forall \phi \in \mathcal{A}_\lambda^\bot, \ \lambda \in \sigma(\Delta_G). \label{eq:agh_system}
    \end{align}
    
    In order to do so we start with some preliminary computations. Let $\phi \in E_\lambda^G$ for some $\lambda \in \sigma(\Delta_G)$. Then
    \begin{align*}
      P(1_T \otimes \phi) &= - \sum_{\ell = 1}^N \vv{Y}_\ell^2 (1_T \otimes \phi)
    \end{align*}
    where $\vv{Y}_\ell$ is given by~\eqref{Eq:def_y_ell}. Explicitly,
    \begin{align*}
      \vv{Y}_\ell^2 (1_T \otimes \phi) &= \sum_{p, p' = 1}^{m^\ell} (\alpha_{\ell p'} \alpha_{\ell p}) \otimes \left( \vv{L}_{p'}^\ell \vv{L}_p^\ell \phi \right) + \sum_{p = 1}^{m^\ell} \left( \vv{W}_\ell \alpha_{\ell p} \right) \otimes \vv{L}_p^\ell \phi .
    \end{align*}
    Let us compute their mixed Sobolev norms: for $j, k \in \Z_+$, since $\Delta_G$ commutes with each $\vv{L}_p^\ell$ we have that
    \begin{multline*}
      (I + \Delta_T)^j (I + \Delta_G)^k \vv{Y}_\ell^2 (1_T \otimes \phi) \\
      = (1 + \lambda)^k \bigg\{ \sum_{p, p' = 1}^{m^\ell} [ (I + \Delta_T)^j(\alpha_{\ell p'} \alpha_{\ell p}) ] \otimes ( \vv{L}_{p'}^\ell \vv{L}_p^\ell \phi ) + \sum_{p = 1}^{m^\ell} [ (I + \Delta_T)^j ( \vv{W}_\ell \alpha_{\ell p} ) ] \otimes \vv{L}_p^\ell \phi \bigg\}
    \end{multline*}
    from where one easily infers that
    \begin{align*}
      \| \vv{Y}_\ell^2 (1_T \otimes \phi) \|_{\sob^{j,k}(T \times G)} &\leq (1 + \lambda)^k \Big( \sum_{p, p' = 1}^{m^\ell} \| \alpha_{\ell p'} \alpha_{\ell p}  \|_{\sob^j(T)} \| \vv{L}_{p'}^\ell \vv{L}_p^\ell \phi  \|_{L^2(G)} + \sum_{p = 1}^{m^\ell} \| \vv{W}_\ell \alpha_{\ell p}  \|_{\sob^j(T)} \| \vv{L}_p^\ell \phi \|_{L^2(G)} \Big) \\
      &\leq C_\ell (1 + \lambda)^{k + \frac{1}{2}}  \sum_{p = 1}^{m^\ell} \Big( \sum_{p' = 1}^{m^\ell}  \| \alpha_{\ell p'} \alpha_{\ell p}  \|_{\sob^j(T)}  +  \| \vv{W}_\ell \alpha_{\ell p}  \|_{\sob^j(T)}  \Big) \| \vv{L}_p^\ell \phi \|_{L^2(G)}
    \end{align*}
    thanks to a simple application of~\cite[Lemma~6.6]{afr20}: the constant $C_\ell > 0$ depends only on $\vv{L}_1^\ell, \ldots, \vv{L}_{m^\ell}^\ell$. Therefore
    \begin{align} \label{eq:PestimadoporLs}
      \| P (1_T \otimes \phi) \|_{\sob^{j,k}(T \times G)} \leq \sum_{\ell = 1}^N  \| \vv{Y}_\ell^2 (1_T \otimes \phi) \|_{\sob^{j,k}(T \times G)} \leq C (1 + \lambda)^{k + \frac{1}{2}} \bigg( \sum_{\ell = 1}^N \sum_{p = 1}^{m^\ell} \| \vv{L}_p^\ell \phi \|_{L^2(G)}^2 \bigg)^{\frac{1}{2}}
    \end{align}
    where $C > 0$ depends on $j$, the vector fields $\vv{Y}_\ell$ and several dimensional constants, but not on $k$, $\lambda$ or $\phi$.

    By Lemma~\ref{lem:hormander_ineq}  there are $j, k\in \Z_+$ and $C' > 0$ with the property that (recall that $P$ in~\eqref{eq:Pdef} equals its own transpose)
    \begin{align} \label{lem:hormander_ineq'}
      \left|\int_{T \times G} f g \ \dd V \right| &\leq C' \| f \|_{\sob^{j}(T \times G)} \| P g\|_{\sob^{k}(T \times G)},
    \end{align}
    for every $g \in \cinfty(T \times G)$ and every $f \in E$~\eqref{eq:E}. Now if, for a given $\lambda \in \sigma(\Delta_G)$, we take $\phi \in \mathcal{A}_\lambda^\bot$ then for a smooth $v \in \ker P$ we use Proposition \ref{Pro:kernel} to obtain $\chi \in \mathcal{A}_\lambda$ such that $\mathcal{F}_\lambda^G(v) = 1_T \otimes \chi$, and thus
    \begin{align*}
      \langle 1_T \otimes \phi, v \rangle_{L^2(T \times G)} = \langle 1_T \otimes \phi, \mathcal{F}_\lambda^G(v) \rangle_{L^2(T \times G)} = \langle 1_T \otimes \phi, 1_T \otimes \chi \rangle_{L^2(T \times G)} = \langle \phi, \chi \rangle_{L^2(G)} = 0
    \end{align*}
    that is, $f \dfn 1_T \otimes \bar{\phi} \in E$. Taking $g \dfn 1_T \otimes \phi$ in~\eqref{lem:hormander_ineq'} together with~\eqref{eq:PestimadoporLs} yield
    \begin{align*}
      \| \phi \|_{L^2(G)}^2 &\leq C' \| 1_T \otimes \bar{\phi} \|_{\sob^{j}(T \times G)} \| P (1_T \otimes \phi)\|_{\sob^{ k}(T \times G)} \\
      &\leq C' \| 1_T \otimes \phi \|_{\sob^{j,j}(T \times G)} \| P (1_T \otimes \phi)\|_{\sob^{ k,k}(T \times G)} \\
      &\leq C' (1 + \lambda)^{j} \| \phi \|_{L^2(G)} C (1 + \lambda)^{k + \frac{1}{2}} \bigg( \sum_{\ell = 1}^N \sum_{p = 1}^{m^\ell} \| \vv{L}_p^\ell \phi \|_{L^2(G)}^2 \bigg)^{\frac{1}{2}}
    \end{align*}
    (where we used~\eqref{eq:mixed_sobolev_open}-\eqref{eq:sobolev_comparation} in the second inequality) which proves~\eqref{eq:agh_system} since
    \begin{align*}
      \bigg( \sum_{\ell = 1}^N \sum_{p = 1}^{m^\ell} \| \vv{L}_p^\ell \phi \|_{L^2(G)}^2 \bigg)^{\frac{1}{2}} &\geq (C C')^{-1} (1 + \lambda)^{-j - k - \frac{1}{2}} \| \phi \|_{L^2(G)}, \quad \forall \phi \in \mathcal{A}_\lambda^\bot, \ \lambda \in \sigma(\Delta_G).
    \end{align*}
  \end{proof}
\end{Prop}


\subsection*{Sufficiency of $\mathcal{L}$ to be~$\mathrm{(AGH)}$ for almost global hypoellipticity of $P$}

The proof of this part will be done in several steps. We first observe that almost global hypoellipticity of $\mathcal L$ yields an \emph{a priori} inequality.
\begin{Thm}\label{Thm:LGHimpliesPGH} If $\mathcal{L}$ is~$\mathrm{(AGH)}$ then there exist $C, \rho > 0$ such that
  \begin{align*}
    \|P \psi\|_{L^2(T \times G)} &\geq C(1+\lambda)^{-\rho} \|\psi\|_{L^2(T \times G)}
  \end{align*}
  for every $\psi \in \cinfty(T; \mathcal{A}_\lambda^{\bot})$ and every $\lambda \in \sigma(\Delta_G)$.
\end{Thm}
\begin{proof} Under our hypothesis~\eqref{it:thm15_hyp1}, for each $\ell \in \{1, \ldots, N\}$ the set $\mathcal{L}_\ell$~\eqref{eq:Lell} is an Abelian subalgebra of $\gr{g}$, with basis $\LL_1^\ell, \ldots, \LL_{m^\ell}^\ell$. Given $\lambda \in \sigma(\Delta_G)$, these vector fields act as a family of commuting, skew-symmetric -- hence normal -- linear endomorphisms of $E_\lambda^G$; all of them are zero on $\mathcal{A}_\lambda$ (by definition) and preserve $\mathcal{A}_\lambda^\bot$ (as we have seen in the proof of Proposition~\ref{Prop:invariance-cluster}). We can then simultaneously diagonalize them on $E_\lambda^G$ as follows: we select
  \begin{itemize}
  \item $\chi_1^{\lambda, \ell}, \ldots, \chi_{c_\lambda}^{\lambda, \ell}$ an orthonormal basis of $\mathcal{A}_\lambda$ (where $c_\lambda \dfn \dim \mathcal{A}_\lambda$ could be even zero) and
  \item $\varphi_{c_\lambda + 1}^{\lambda, \ell}, \ldots, \varphi_{d_\lambda^G}^{\lambda, \ell}$ an orthonormal basis of $\mathcal{A}_\lambda^\bot$ that are common eigenvectors to $\LL_1^\ell, \ldots, \LL_{m^\ell}^\ell$.
  \end{itemize}
  Their eigenvalues are purely imaginary i.e.
  \begin{align*} 
    \vv{L}_p^\ell \varphi_i^{\lambda, \ell} &= \sqrt{-1} \gamma_{i,p}^{\lambda, \ell}  \varphi_i^{\lambda, \ell}, \quad \gamma_{i,p}^{\lambda, \ell} \in \R, \ p \in \{1, \ldots, m^\ell\}, \ i \in \{c_\lambda + 1, \ldots, d^G_\lambda\}
  \end{align*}
  and also satisfy $\gamma_{i,p}^{\lambda, \ell} = \mathrm{O}(\sqrt{\lambda})$~\cite[Lemma~6.6]{afr20}. Applying $\mathfrak{a}_\ell(t)$ to $\varphi_i^{\lambda, \ell}$ we obtain
  \begin{align*}
    \mathfrak{a}_\ell(t) \varphi_i^{\lambda, \ell} = \sum_{p=1}^{m^\ell} \alpha_{\ell p}(t) \vv{L}^\ell_p \varphi_i^{\lambda, \ell} = \sqrt{-1} \sum_{p=1}^{m^\ell} \alpha_{\ell p}(t) \gamma_{i,p}^{\lambda, \ell} \varphi_i^{\lambda, \ell}.
  \end{align*}

  
  A general $\psi\in \cinfty(T; \mathcal{A}_\lambda^{\bot})$ is written as
  \begin{align*}
    \psi &= \sum_{i=c_\lambda+ 1}^{d^{G}_\lambda} \psi_{i}^\ell \otimes \varphi^{\lambda, \ell}_i. 
  \end{align*}
  If we define $D_\ell: T \times \R^{m^\ell} \rarr \R$ by
  \begin{align*}
    D_\ell (t, \gamma) &\dfn \Big( \sum_{p = 1}^{m^\ell} \alpha_{\ell p}(t) \gamma_p \Big)^2, \quad t \in T, \ \gamma \in \R^{m^\ell},
  \end{align*}
  we obtain
  \begin{align}\label{est_a_ell_aux}
    \| \mathfrak{a}_\ell(t) \psi(t)\|_{L^{2}(G)}^2= \sum_{i=c_\lambda+ 1}^{d^G_\lambda} |\psi^\ell_i(t)|^2 D_{\ell}(t, \gamma^{\lambda, \ell}_i)
  \end{align}
  where $\gamma^{\lambda, \ell}_i \in \R^{m^\ell}$ is defined in the obvious way. The fact that $\mathcal{L}$ is~$\mathrm{(AGH)}$ allows us to use~\eqref{eq:agh_system} to obtain
\begin{align*}
  \|\psi(t)\|^2_{L^2(G)} \leq C^{-2}(1 + \lambda)^{2\rho} \sum_{\ell=1}^N \sum_{p = 1}^{m^\ell} \|\vv{L}_p^\ell \psi(t) \|_{L^2(G)}^2 = C^{-2}(1 + \lambda)^{2\rho} \sum_{\ell=1}^N \sum_{p = 1}^{m^\ell}\sum_{i=c_\lambda+1}^{d^G_\lambda} |\psi_i^\ell(t)|^2|\gamma^{\lambda, \ell}_{i, p}|^2
\end{align*}
for each $t \in T$. Integrating both sides over $T$ yields
\begin{align}\label{eq:psiestimate}
  \|\psi\|^2_{L^{2}(T \times G)}\leq C^{-2}(1 + \lambda)^{2\rho} \sum_{\ell=1}^N \sum_{i=c_\lambda+1}^{d^G_\lambda} \int_T|\psi_i^\ell(t)|^2|\gamma^{\lambda, \ell}_{i}|^2 \dd V_T(t).
\end{align}

Adapting the proof of~\cite[Lemma~3.1]{bfp17} one checks that there are $\alpha, \delta>0$ such that for every $\ell \in \{1, \ldots, N\}$ and every $i \in \{c_\lambda + 1,\ldots, d_\lambda^G\}$ there is an open set $A_{i}^{\lambda, \ell} \subset T$ with $\mathrm{vol}(A_i^{\lambda, \ell}) \geq \delta$ and such that 
\begin{align*}
  \alpha|\gamma^{\lambda, \ell}_{i}|^2 &\leq D_\ell(t, \gamma^{\lambda, \ell}_{i}), \quad \forall t\in A_{i}^{\lambda, \ell}
\end{align*}
and by~\cite[Proposition~6.2]{afr20} -- a convenient substitute of~\cite[eqn.~(2.10)]{hp00} for arbitrary compact Riemannian manifolds -- we can find a constant $C_1>0$ depending only on $\delta$ such that
\begin{align*}
  \int_T|\psi_i^\ell(t)|^2|\gamma^{\lambda, \ell}_{i}|^2 \dd V_T(t)&\leq C_1\bigg(\int_{A_i^{\lambda, \ell}}|\psi_i^\ell(t)|^2|\gamma_{i}^{\lambda, \ell}|^2 \dd V_T(t)+ \| |\gamma_i^{\lambda, \ell}| \dd_T\psi_i^{\ell}\|^2_{L^2(T)}\bigg)\\
                                                                   &\leq C_1\bigg(\alpha^{-1}\int_{A_i^{\lambda, \ell}}|\psi_i^\ell(t)|^2D_\ell(t,\gamma^{\lambda, \ell}_{i}) \dd V_T(t)+ B_\ell \lambda\|  \dd_T\psi_i^{\ell}\|^2_{L^2(T)}\bigg),  
\end{align*}
where $B_{\ell}>0$ depends only on $\vv{L}_1^\ell,\ldots, \vv{L}_{m^\ell}^\ell$. Summing both sides over $i$ and $\ell$, using~\eqref{eq:psiestimate}, \eqref{est_a_ell_aux} and~\eqref{Pro:kernel-eq1} we obtain
\begin{align*}
  \|\psi\|^2_{L^{2}(T \times G)} &\leq C^{-2}(1 + \lambda)^{2\rho}\sum_{\ell=1}^N\sum_{i=c_\lambda+1}^{d^G_\lambda}\int_T|\psi_i^\ell(t)|^2|\gamma^{\lambda, \ell}_{i}|^2 \dd V_T(t) \\
  &\leq C_2(1 + \lambda)^{2\rho+1} \bigg( \sum_{\ell=1}^N \|\mathfrak{a}_\ell(t, \vv{X}) \psi\|_{L^2(T \times G)}^2+  \langle\Delta_T \psi, \psi\rangle_{L^2(T \times G)}\bigg)\\
  &\leq C_3(1 + \lambda)^{2\rho+1}\bigg(  \sum_{\ell=1}^N\|(\mathfrak{a}_\ell(t, \vv{X})+ \vv{W}_\ell) \psi\|_{L^2(T \times G)}^2+  \langle\Delta_T \psi, \psi\rangle_{L^2(T \times G)}\bigg) \\
  &= C_3(1+\lambda)^{2 \rho+1}\langle P\psi, \psi\rangle_{L^2(T \times G)}\\
  &\leq C_3(1+\lambda)^{2 \rho+1}\| P \psi\|_{L^2(T \times G)} \| \psi\|_{L^2(T \times G)}
\end{align*}
where we argue as in~\cite[pp.~23--24]{afr20} to incorporate $\vv{W}_\ell$ in the third inequality. This finishes the proof.
\end{proof}
\begin{Rem} Theorem~\ref{Thm:LGHimpliesPGH} is the only place where property~\eqref{it:thm15_hyp1} is used in our arguments.
\end{Rem}

For our next proposition we need to combine two estimates on the Fourier projections of a distribution to conclude its smoothness. First we recall~\cite[Corollary~2.10]{afr20}, whose statement we transcribe below.
\begin{Prop} \label{cor:partial_smoothness_converse} If $f \in \D'(T \times G)$ is such that
  \begin{enumerate}
  \item for every $s > 0$ there exists $C > 0$ such that
  \begin{align*}
    \| \mathcal{F}^G_\lambda(f) \|_{L^2(T \times G)} &\leq C (1 + \lambda)^{-s}, \quad \forall \lambda \in \sigma(\Delta_G) 
  \end{align*}
 holds and
  \item for every $s' > 0$ there exist $C' > 0$ and $\theta \in (0, 1)$ such that
    \begin{align*}
      \| \mathcal{F}^T_\mu \mathcal{F}^G_\lambda (f) \|_{L^2(T \times G)} &\leq C' (1 + \mu + \lambda)^{-s'}, \quad \forall (\mu, \lambda) \in \Lambda_\theta, 
    \end{align*}
    where
    \begin{align}
      \Lambda_\theta &\dfn \{ (\mu, \lambda) \in \sigma(\Delta_T) \times \sigma(\Delta_G) \st (1 + \lambda) \leq (1 + \mu)^\theta \}, \label{eq:Atheta}
    \end{align}
  \end{enumerate}
  then $f \in \cinfty(T \times G)$.
\end{Prop}

\begin{Prop}\label{Pro:GHaperpimplyGSaperp} If $\mathcal{L}$ is $\mathrm{(AGH)}$ in $G$ then $P$ is $\mathrm{(GH)}_{\mathcal{A}^\bot}$. 
\end{Prop}
\begin{proof} Let us consider $u\in \D'_{\mathcal{A}^{\bot}}(T \times G)$ such that $P u\in \cinfty_{\mathcal{A}^{\bot}}(T \times G)$.
  Since $P u\in \cinfty(T \times G)$ and $\tilde{P}$ is elliptic we have, thanks to~\cite[Corollary~4.2]{afr20}, that $\mathcal{F}^G_\lambda(u)$ is smooth -- hence actually belongs to $\cinfty(T; \mathcal{A}^{\bot}_\lambda)$ -- for every $\lambda\in \sigma(\Delta_G)$. Moreover, from~\cite[Corollary~4.4]{afr20} we have that for every $s>0$ there are $C_s>0$ and $\theta\in (0,1)$ such that
  \begin{align*}
    \|\mathcal{F}_\mu^T \mathcal{F}^G_\lambda(u)\|_{L^2(T \times G)} &\leq C_s(1+ \mu +\lambda)^{-s}
  \end{align*}
  for every $(\mu, \lambda)\in \Lambda_\theta$.  It follows from Theorem~\ref{Thm:LGHimpliesPGH} applied to $\psi = \mathcal{F}^G_\lambda(u)$ that
  \begin{align*}
    \|\mathcal{F}^G_\lambda(u)\|_{L^2(T \times G)} &\leq C^{-1}(1+\lambda)^{\rho} \|\mathcal{F}^G_\lambda(Pu)\|_{L^2(T \times G)}.
  \end{align*}
  Now~\cite[Corollary 2.9]{afr20} ensures that for every $s>0$ there exists $C_s'>0$ such that
  \begin{align*}
    \|\mathcal{F}^G_\lambda(Pu)\|_{L^2(T \times G)} &\leq C_s'(1+\lambda)^{-s-\rho}
  \end{align*}
  for every $\lambda\in \sigma(\Delta_G)$ and therefore
  \begin{align*}
    \|\mathcal{F}^G_\lambda(u)\|_{L^2(T \times G)} &\leq C^{-1}C_s'(1+\lambda)^{-s}.
  \end{align*}
  It follows from Proposition~\ref{cor:partial_smoothness_converse} that $u$ is smooth.   
\end{proof}

Before we can finally conclude the proof of Theorem~\ref{Thm:main-P} let us recall an application of a classical result on elliptic operators. Consider an elliptic LPDO $Q$ of order $m_0 \geq 1$ on $T$. It is well known (see e.g.~\cite[Theorem~6.29]{warner_manifolds}) that given $k\in\Z_+$ there exists $C >0$ such that
\begin{align} \label{elip_est_1}
  \|v\|_{\sob^{k+m_0}(T)} &\leq C (\|Q v\|_{\sob^k(T)}+\|v\|_{\sob^k(T)}), \quad \forall v \in \sob^{k+m_0}(T).
\end{align}
We can obtain a more precise inequality in the $L^2$-orthogonal of $L^2(T) \cap \ker Q$ -- the latter is a closed subspace of $L^2(T)$.
\begin{Lem} \label{lemwarner} \label{elip_est} Given $k\in\Z_+$ there exists $C >0$ such that
  \begin{align} \label{elip_est_2}
    \|v\|_{\sob^{k+m_0}(T)} &\leq C \|Qv \|_{\sob^k(T)}, \quad \forall v \in \left(L^2(T) \cap \ker Q\right)^\bot \cap \sob^{k+m_0}(T).
  \end{align}
\end{Lem}
\begin{proof} If~\eqref{elip_est_2} does not hold then we can find a sequence $\{v_\nu \}_{\nu \in \N} \sset (L^2(T) \cap \ker Q)^\bot \cap \sob^{k+m_0}(T)$ such that $\|v_\nu\|_{\sob^{k+m_0}(T)} = 1$ for every positive integer $\nu$ and
  \begin{align} \label{elip_est_3}
    \|Q v_\nu \|_{\sob^k(T)} &\leq \frac{1}{\nu}. 
  \end{align}
  By passing to a subsequence if necessary we can also assume, by Rellich's Lemma, that there exists $v \in \sob^k(T)$ such that $v_\nu \to v$ in $\sob^k(T)$. In particular,  it follows that $v_\nu \to v$ in $L^2(T)$ and since $v_\nu \in (L^2(T) \cap \ker Q)^\bot$, which is a closed subspace of $L^2(T)$, we have that $v \in (L^2(T) \cap \ker Q)^\bot$. Now we have, from~\eqref{elip_est_3}, that $Qv_\nu \to 0$ in $\D'(T)$ and since $v_\nu \to v$ in $\D'(T)$ it follows from the continuity of $Q$ in $\D'(T)$ that $Qv = 0$, so $v \in L^2(T) \cap \ker Q$. Hence $v=0$ and~\eqref{elip_est_1}-\eqref{elip_est_3} give
  \begin{align*}
    1 \leq C (\|Qv_\nu\|_{\sob^k(T)} + \|v_\nu\|_{\sob^k(T)}) \longrightarrow 0 \quad \text{when $\nu \to \infty$}.           
  \end{align*}
\end{proof}

\begin{Thm} If $P$ is $(\mathrm{GH})_{\mathcal{A}^\bot}$ then $P$ is~$\mathrm{(AGH)}$ in $T \times G$.
\end{Thm}
\begin{proof} Consider $u\in \D'(T\times G)$ such that  $f \dfn Pu\in \cinfty(T \times G)$: by Proposition~\ref{Prop:invariance-cluster} we are allowed to write 
  \begin{align*}
    P \pi_{\mathcal{A}} (u) = \pi_{\mathcal{A}} (f) \in \cinfty_{\mathcal{A}}(T \times G), &\quad P \pi_{\mathcal{A}^{\bot}} (u) = \pi_{\mathcal{A}^{\bot}} (f) \in \cinfty_{\mathcal{A}^{\bot}}(T \times G)
  \end{align*}
so by assumption we have $\pi_{\mathcal{A}^{\bot}} (u) \in \cinfty_{\mathcal{A}^{\bot}}(T \times G)$, thus reducing our problem to find $v \in \cinfty(T \times G)$ such that $Pv = P \pi_{\mathcal{A}} (u) = \pi_{\mathcal{A}}(f)$.

Expanding the partial Fourier coefficients of $u$ in terms of orthonormal bases adapted to $\mathcal{A}, \mathcal{A}^\bot$ as we did in Section~\ref{sec:cluster_P} we have
\begin{align*}
  \mathcal{F}_\lambda^G(u) = \sum_{i = 1}^{c_\lambda} \mathcal{F}_\lambda^G(u)_i \otimes \chi_i^\lambda + \sum_{i = c_\lambda + 1}^{d_\lambda^G} \mathcal{F}_\lambda^G(u)_i \otimes \varphi_i^\lambda = \mathcal{F}^{G,\mathcal{A}}_\lambda(u) + \mathcal{F}^{G,\mathcal{A}^\bot}_\lambda(u)
\end{align*}
and recalling that $P = \tilde{P}$ on $\D'_{\mathcal{A}}(T \times G)$ (see~\eqref{eq:P_tildeP_cluster}) we have that
\begin{align*}
  \sum_{i=1}^{c_\lambda} \left(\tilde P \mathcal F_\lambda^G(u)_i \right) \otimes \chi^\lambda_i = \tilde{P} \mathcal{F}_\lambda^{G,\mathcal{A}} (u)&=  P \mathcal{F}_\lambda^{G,\mathcal{A}} (u)= \mathcal{F}^{G,\mathcal{A}}_\lambda(f)= \sum_{i=1}^{c_\lambda} \mathcal F_\lambda^G(f)_i \otimes \chi^\lambda_i, 
\end{align*}
from which we obtain, by uniqueness of the coefficients, that
\begin{align*}
  \tilde P \mathcal F_\lambda^G(u)_i &= \mathcal F_\lambda^G(f)_i \in \cinfty(T), \quad \forall i \in \{1,\ldots,c_\lambda\}.
\end{align*}
Since $\tilde{P}$ is elliptic it follows that $\mathcal F_\lambda^G(u)_i \in \cinfty(T)$ for every $i \in \{1, \ldots, c_\lambda\}$. Also, $\ker \tilde{P}$ is closed in $L^2(T)$, thus inducing a decomposition into orthogonal direct sum
\begin{align*}
  \mathcal F_\lambda^G(u)_i = w_i^\lambda+v_i^\lambda, \quad \text{where $w_i^\lambda \in \ker \tilde P$ and $v_i^\lambda \in (\ker \tilde P)^\bot$}
\end{align*}
where $v^\lambda_i$ is smooth and satisfies
\begin{align*}
  \tilde{P} v^\lambda_i &= \mathcal F_\lambda^G(f)_i, \quad \forall i \in \{1,\ldots,c_\lambda\}.
\end{align*}
We claim that
\begin{align*}
  v &\dfn \sum_{\alpha \in \sigma(\Delta_M)} \sum_{\mu + \lambda = \alpha} \sum_{j = 1}^{d_\mu^T} \sum_{i = 1}^{c_\lambda} \langle v^\lambda_i, \psi_j^\mu \rangle_{L^2(T)} \psi_j^\mu \otimes \chi_i^\lambda \quad \left( = \sum_{\lambda \in \sigma(\Delta_G)} \sum_{i = 1}^{c_\lambda} v^\lambda_i \otimes \chi_i^\lambda \right)
\end{align*}
converges in $\cinfty(T \times G)$, in which case we certainly have $v \in \cinfty_{\mathcal{A}}(T \times G)$ hence
\begin{align*}
  \mathcal{F}_\lambda^G(P v) &= \tilde{P} \mathcal{F}_\lambda^G (v) \\
  &= \tilde{P} \sum_{\mu \in \sigma(\Delta_T)} \sum_{j = 1}^{d_\mu^T} \sum_{i = 1}^{c_\lambda} \langle v^\lambda_i, \psi_j^\mu \rangle_{L^2(T)} \psi_j^\mu \otimes \chi_i^\lambda \\
  &= \tilde{P} \sum_{i = 1}^{c_\lambda} v^\lambda_i \otimes \chi_i^\lambda \\
  &= \sum_{i = 1}^{c_\lambda} (\tilde{P} v^\lambda_i) \otimes \chi_i^\lambda \\
  &= \sum_{i = 1}^{c_\lambda}  \mathcal F_\lambda^G(f)_i \otimes \chi_i^\lambda \\
  &= \mathcal{F}_\lambda^G(\pi_{\mathcal{A}}(f))
\end{align*}
for every $\lambda \in \sigma(\Delta_G)$, and therefore $Pv = \pi_{\mathcal{A}}(f)$ -- thus concluding our proof. In order to prove the claim, we appeal to Corollary~\ref{cor:charac_smoothness_fourier_proj}: thanks to Lemma~\ref{lemwarner}, for each $k \in \Z_+$ there exists $C_k > 0$ such that
\begin{align*}
  \|v^\lambda_i\|_{\sob^{k + 2}(T)} &\leq C_k \| \mathcal F_\lambda^G(f)_i \|_{\sob^k(T)}, \quad \forall i \in \{1,\ldots,c_\lambda\}, \ \lambda \in \sigma(\Delta_G)
\end{align*}
where
\begin{align*}
  \|v^\lambda_i\|_{\sob^{k + 2}(T)}^2 &= \sum_{\mu \in \sigma(\Delta_T)} (1 + \mu)^{2k + 4} \sum_{j = 1}^{d_\mu^T} | \langle v_i^\lambda, \psi_j^\mu \rangle_{L^2(T)}|^2
\end{align*}
and
\begin{align*}
  \| \mathcal F_\lambda^G(f)_i \|_{\sob^k(T)}^2 &= \sum_{\mu \in \sigma(\Delta_T)} (1 + \mu)^{2k} \sum_{j = 1}^{d_\mu^T} \left| \left \langle \mathcal{F}_\lambda^G(f)_i, \psi_j^\mu \right \rangle_{L^2(T)} \right|^2 \\
  &= \sum_{\mu \in \sigma(\Delta_T)} (1 + \mu)^{2k} \sum_{j = 1}^{d_\mu^T} | \langle f, \psi_j^\mu \otimes \chi_i^\lambda \rangle_{L^2(T \times G)} |^2
\end{align*}
and since $f$ is smooth we have by Corollary~\ref{cor:charac_smoothness_fourier_proj} that there exists $C_k' > 0$ such that
\begin{align*}
  | \langle f, \psi_j^\mu \otimes \chi_i^\lambda \rangle_{L^2(T \times G)} | &\leq C_k' (1 + \mu + \lambda)^{-2k -n}, \quad \forall i,j,\mu, \lambda
\end{align*}
hence
\begin{align*}
  \| \mathcal F_\lambda^G(f)_i \|_{\sob^k(T)}^2 \leq (C'_k)^2 (1 + \lambda)^{-2k} \sum_{\mu \in \sigma(\Delta_T)} d_\mu^T(1 + \mu)^{-2n} \leq (C''_k)^2 (1 + \lambda)^{-2s}
\end{align*}
thanks to Weyl's asymptotic formula. Therefore
\begin{align*}
 | \langle v^\lambda_i, \psi_j^\mu \rangle_{L^2(T)} | \leq (1 + \mu)^{-k - 2} \| v^\lambda_i\|_{\sob^{k + 2}(T)} \leq C_k C''_k (1 + \mu)^{-k - 2} (1 + \lambda)^{-k} \leq C_k C''_k (1 + \mu + \lambda)^{-k}
\end{align*}
for every $i,j,\mu, \lambda$, so by Corollary~\ref{cor:charac_smoothness_fourier_proj} we conclude that $v \in \cinfty(T \times G)$.

\end{proof}

\section{Example: characterization of globally solvable operators on torus} \label{sec:examples}

Notwithstanding the applicability of our results to general compact Lie groups, on which one may hope to find many families of~$\mathrm{(AGH)}$ vector fields satisfying the hypotheses of Theorem~\ref{Thm:main-P} (even single ones, as we have seen in Example~\ref{exa:su2}), the special case $G= \mathbb{T}^m$ is of particular interest. While in this case hypothesis~\eqref{it:thm15_hyp1} is clearly superfluous, historically this kind of problem is studied on tori: in the work~\cite{petronilho02}, which inspired the questions we investigated in the present paper, the author analyzes global solvability of the following differential operator defined in $\mathbb T_t^n\times \mathbb T_x^m$
\begin{align*}
  P &= -\sum_{k = 1}^n \del_{t_k}^2 - \bigg(\sum_{j=1}^m a_j(t)\partial_{x_j}\bigg)^2  
\end{align*}
where $a_1,\ldots,a_m \in \cinfty(\mathbb T^n;\R)$, and then characterizes such a property in terms of Diophantine conditions.

Let us consider the more general operator
\begin{align}
  P &\dfn \Delta_T - \sum_{\ell = 1}^N \Big( \sum_{j = 1}^m a_{\ell j}(t) \del_{x_j} + \vv{W}_\ell \Big)^2 \label{eq:Pdef_torus}
\end{align}
defined in $T \times \mathbb{T}^m$ where $T$ will remain a compact Riemannian manifold enjoying our usual assumptions. That is, we are going to specialize our discussion of our general operator~\eqref{eq:Pdef} in the case when $G$ is the $m$-dimensional torus. We recall that $\sigma(\Delta_{\mathbb{T}^m})=\{ n^2 \st n\in \Z_+\}$ and an orthonormal basis for $E^{\mathbb{T}^m}_{n^2}$ is given by the exponentials $x \in \mathbb{T}^m \mapsto e^{i x\xi} \in \C$ for $\xi\in \Z^m$ with $|\xi|^2= n^2$.

Below, for each $\ell \in \{1, \ldots, N\}$ we will assume for simplicity that $m^\ell \dfn \dim \mathcal{L}_\ell \in \{1, \ldots, m - 1\}$: the extreme cases $m^\ell = 0$ and $m^\ell = m$ -- which correspond, as we will see, to all the coefficients $a_{\ell 1}, \ldots, a_{\ell m}$ being zero, or linearly independent, respectively -- are treated analogously provided one interprets the notation carefully. A basis for $\mathcal{L}_\ell$ is given by
\begin{align*}
  \mathrm{L}_p^\ell &\dfn \partial_{x_{j_p^\ell}} + \sum_{q=1}^{d^\ell} \lambda^\ell_{qp} \partial_{x_{i^\ell_q}}, \quad p \in \{1,\ldots,m^\ell\},
\end{align*}
where the coefficients $\lambda_{qp}^\ell \in \R$ are obtained as follows: let $a_{\ell j_1^\ell}, \ldots, a_{\ell j_{m^\ell}^\ell}$ be a basis of $\Span_\R \{a_{\ell 1}, \ldots, a_{\ell m}\}$ and write, for the remaining indices $1 \leq i_1^\ell < \cdots < i_{d^\ell}^\ell \leq m$ -- where $d^\ell \dfn m - m^\ell$:
\begin{align*}
  a_{\ell i_q^\ell} &= \sum_{p = 1}^{m^\ell} \lambda_{qp}^\ell a_{\ell j_p^\ell}, \quad q \in \{1, \ldots, d^\ell \}.
\end{align*}
Our spectral cluster~\eqref{spectralclusterqueusamos} adapted to $P$ is then given by
\begin{align*}
  \mathcal{A}_{n^2} &= E^{\mathbb{T}^m}_{n^2} \cap \big(\bigcap_{\ell=1}^N \bigcap_{p=1}^{m^\ell} \ker \mathrm{L}^\ell_p \big), \quad n \in \Z_+.
\end{align*}
On time, recall that for $\xi=(\xi_1,\ldots,\xi_m) \in \Z^m$ we have that
\begin{align} \label{system_app_eigenvec}
  \mathrm{L}_p^\ell e^{ix\xi} = i\bigg(\xi_{j^\ell_p} + \sum_{q=1}^{d^\ell} \lambda^\ell_{qp} \xi_{i^\ell_q}\bigg)e^{ix\xi}, \quad \forall \ell \in \{1,\ldots,N\}, \ p \in \{1,\ldots,m^\ell\}
\end{align}
which in turn naturally leads one to consider the following set that will be important for us:
\begin{align*}
  \Gamma &\dfn \Big\{\xi=(\xi_1,\ldots,\xi_m) \in \Z^m \st \xi_{j^\ell_p} + \sum_{q=1}^{d^\ell} \lambda^\ell_{qp} \xi_{i^\ell_q} = 0, \quad \forall \ell \in \{1,\ldots,N\}, \ p \in \{1,\ldots,m^\ell\} \Big\}.  
\end{align*}

A simple computation shows the next result, so its proof will be omitted.
\begin{Lem} \label{charact_ker_perp} If $n \in\Z_+$ then
  \begin{align*}
    \mathcal{A}_{n^2} &= \Span_\C \left\{e^{ix\xi} \st \text{$|\xi|=n$ and $\xi \in \Gamma$} \right\}  
  \end{align*}
  and
  \begin{align*}
    \mathcal{A}_{n^2}^\bot &= \Span_\C \left\{e^{ix\xi} \st  \text{$|\xi|=n$ and $\xi \notin \Gamma$} \right\}.
  \end{align*}
\end{Lem}
The almost global hypoellipticity of $\mathcal{L}$ in $G$ is characterized by~\eqref{eq:agh_system}, a condition that, in the torus, may be expressed in terms of traditional Diophantine inequalities.
\begin{Prop} Inequality \eqref{eq:agh_system} is equivalent to the existence of $C',\rho'>0$ such that
 \begin{align} \label{agh_ineq2}
   \bigg(\sum_{\ell=1}^N\sum_{p=1}^{m^\ell}\big|\xi_{j^\ell_p} + \sum_{q=1}^{d^\ell} \lambda^\ell_{qp} \xi_{i^\ell_q}\big|^2\bigg)^{\frac{1}{2}} &\geq C'(1+|\xi|^2)^{-\rho'}, \quad \forall \xi \in \Z^m\setminus \Gamma.
 \end{align}
\end{Prop}
\begin{proof} Suppose that~\eqref{eq:agh_system} holds true. If $\xi \in \Z^m\setminus \Gamma$ then we set $n \dfn |\xi|$ and from Lemma~\ref{charact_ker_perp} we have that $e^{ix\xi} \in \mathcal{A}^\bot_{n^2}$, hence~\eqref{system_app_eigenvec} and~\eqref{eq:agh_system} imply that
  \begin{align*}
    \sum_{\ell=1}^N\sum_{p=1}^{m^\ell}\big|\xi_{j^\ell_p} + \sum_{q=1}^{d^\ell} \lambda^\ell_{qp} \xi_{i^\ell_q}\big|^2 &= \sum_{\ell=1}^N\sum_{p=1}^{m^\ell} \Big\|i\Big(\xi_{j^\ell_p} + \sum_{q=1}^{d^\ell} \lambda^\ell_{qp} \xi_{i^\ell_q}\Big)e^{ix\xi}\Big\|^2_{L^2(\mathbb T^m)} \\
    &= \sum_{\ell=1}^N\sum_{p=1}^{m^\ell} \Big\|\mathrm{L}_p^\ell e^{ix\xi}\Big\|^2_{L^2(\mathbb T^m)}\\
    &\ge C^2(1+n^2)^{-2\rho}\\
    &= C^2(1+|\xi|^2)^{-2\rho}
  \end{align*}
  and we obtain~\eqref{agh_ineq2} with $C'\dfn C$ and $\rho' \dfn \rho$.

  Conversely, for fixed $n\in\Z_+$ we may write (see Lemma~\ref{charact_ker_perp}) every $\phi \in \mathcal{A}_{n^2}^\bot$ as
  \begin{align*}
    \phi &= \sum_{\substack{|\xi|=n\\ \xi\notin\Gamma}} \phi_\xi e^{ix\xi}, \quad \text{where $\phi_\xi \dfn \langle \phi, e^{ix\xi} \rangle_{L^2(\mathbb T^m)}$}.  
  \end{align*}
  By orthogonality of the exponentials
  \begin{align*}
    \|\phi\|_{L^2(\mathbb T^m)}^2 &= \sum_{\substack{|\xi|=n\\ \xi\notin\Gamma}} |\phi_\xi|^2  
  \end{align*}
  and from~\eqref{agh_ineq2} and~\eqref{system_app_eigenvec} it follows that
  \begin{align*}
    \sum_{\ell=1}^N\sum_{p=1}^{m^\ell}\|\mathrm{L}_p^\ell \phi\|_{L^2(\mathbb{T}^m)}^2 &= \sum_{\ell=1}^N\sum_{p=1}^{m^\ell}\big\|\sum_{\substack{|\xi|=n\\ \xi\notin\Gamma}} \phi_\xi \mathrm{L}_p^\ell e^{ix\xi}\big\|_{L^2(\mathbb{T}^m)}^2\\
    &=\sum_{\ell=1}^N\sum_{p=1}^{m^\ell}\big\|\sum_{\substack{|\xi|=n\\ \xi\notin\Gamma}} \phi_\xi \big(\xi_{j^\ell_p} + \sum_{q=1}^{d^\ell} \lambda^\ell_{qp} \xi_{i^\ell_q}\big)e^{ix\xi}\big\|_{L^2(\mathbb{T}^m)}^2\\
    &=\sum_{\ell=1}^N\sum_{p=1}^{m^\ell}\sum_{\substack{|\xi|=n\\ \xi\notin\Gamma}} \big|\phi_\xi\big|^2 \big| \xi_{j^\ell_p} + \sum_{q=1}^{d^\ell} \lambda^\ell_{qp} \xi_{i^\ell_q} \big|^2\\
    &=\sum_{\substack{|\xi|=n\\ \xi\notin\Gamma}} |\phi_\xi|^2 \sum_{\ell=1}^N\sum_{p=1}^{m^\ell} \big| \xi_{j^\ell_p} + \sum_{q=1}^{d^\ell} \lambda^\ell_{qp} \xi_{i^\ell_q} \big|^2\\
    &\ge \sum_{\substack{|\xi|=n\\ \xi\notin\Gamma}} |\phi_\xi|^2 \big(C'(1+|\xi|^2)^{-\rho'}\big)^2\\
    &=  (C')^2\big(1+n^2\big)^{-2\rho'} \|\phi\|^2_{L^2(\mathbb T^m)}
  \end{align*}
  from where the conclusion follows easily.
\end{proof}

In particular, Theorem~\ref{Thm:main-P} generalizes condition~$\mathrm{(III)}_2$ in~\cite[Theorem~1.1]{petronilho02} -- so, together with the results in~\cite{afr20}, we have given a complete characterization of global solvability of the more general operator $P$ in~\eqref{eq:Pdef_torus}, much in the same lines as in Petronilho's work.

\section{Propagation of regularity} \label{sec:singularities}

In this final section we take advantage of the machinery developed so far to prove a result on propagation of regularity for $P$ given by~\eqref{eq:Pdef}. We will continue to assume that $G$ is a Lie group and $T$ is a smooth manifold enjoying the properties stated at the beginning of Section~\ref{sec:partial_exp}, and also carrying Riemannian metrics satisfying the assumptions in Section~\ref{Sec:tube-type}. We stress that hypothesis~\eqref{it:thm15_hyp1} is not required here.

Recall that $\tilde{P}=\Delta_T - \sum_{\ell = 1}^N \vv{W}_\ell^2$ is an elliptic operator in $T$ so~\cite[Corollary~4.2]{afr20} ensures that $\mathcal F_\lambda^G(u) \in \cinfty(T;E_\lambda^G)$ when $u \in \D'(T\times G)$ is such that $Pu \in \cinfty(T\times G)$. In this section we shall consider the following type of ``local condition'' in $T$:
\begin{quote} There is an open subset $U \subset T$ such that for every $\rho\in \R$ we can find $C_\rho>0$ such that
  \begin{align} \label{eq:regnumponto}
    \|\mathcal{F}_{\lambda}^{G}(u)\|_{L^{2}(U\times G)} &\leq C_\rho(1+\lambda)^{-\rho}, \quad \forall \lambda\in \sigma(\Delta_G).
  \end{align}
\end{quote}
In order to prove our main result of this section we need an inequality that is a consequence of~\cite[Proposition~6.2]{afr20}: given an open subset $U\subset T$, there exists $C>0$ such that
\begin{align}
  \|\psi\|^2_{L^2(T)} &\leq C (\left\|\psi\right\|^2_{L^2(U)} + \|\dd_T\psi \|^2_{L^2(T)}  ), \quad \forall \psi \in \cinfty(T). \label{eq:calculus}
\end{align}
Now we can state and prove our main result of this section.
\begin{Thm} \label{thm:propagation} If $u \in \D'(T\times G)$ satisfies $Pu\in \cinfty(T\times G)$ and there exists an open set $U \subset T$ on which condition~\eqref{eq:regnumponto} holds then $u \in \cinfty(T \times G)$.
\end{Thm}
\begin{proof} By~\cite[Corollary~4.4]{afr20} for every $s>0$ there are $C_s>0$ and $\theta\in (0,1)$ such that
  \begin{align}\label{eq:prop-est-logcone}
    \|\mathcal{F}_\mu^T \mathcal{F}^G_\lambda(u)\|_{L^2(T \times G)} &\leq C_s(1+ \mu +\lambda)^{-s}
  \end{align}
  for every $(\mu, \lambda)\in \Lambda_\theta$, where $\Lambda_\theta$ is given by~\eqref{eq:Atheta}. It follows from~\eqref{Pro:kernel-eq1},~\eqref{eq:equality-energy-part} and the Cauchy-Schwarz inequality that
  \begin{align}
    \sum_{j=1}^{d_\lambda^G} \|\dd_T [\mathcal F_\lambda^G (u)_i] \|^2_{L^2(T)} &\leq \|\mathcal{F}_\lambda^G(Pu)\|_{L^{2}(T \times G)}  \| \mathcal{F}_\lambda^G(u)\|_{L^{2}(T\times G)}.\label{eq:menorqueoproduto}
  \end{align}
  By~\eqref{eq:calculus} applied to $\psi = \mathcal F_\lambda^G(u)_i$ we have
  \begin{align*}
    \| \mathcal F_\lambda^G(u)_i \|_{L^2(T)}^2 &\leq C\left(\| \mathcal F_\lambda^G(u)_i \|_{L^2(U)}^2 + \|\dd_T [\mathcal F_\lambda^G(u)_i] \|_{L^2(T)}^2\right),
  \end{align*}
  from where we conclude, also using~\eqref{eq:menorqueoproduto}, that
  \begin{align*}
    \sum_{i=1}^{d_\lambda^G} \| \mathcal F_\lambda^G(u)_i \|_{L^2(T)}^2 &\leq  C\Big(\sum_{i=1}^{d_\lambda^G} \| \mathcal F_\lambda^G(u)_i \|_{L^2(U)}^2 +  \|\mathcal{F}_\lambda^G(Pu)\|_{L^{2}(T \times G)}  \| \mathcal{F}_\lambda^G(u)\|_{L^{2}(T\times G)}\Big),
  \end{align*}
  which in turn gives
  \begin{align}\label{eq:equcominte}
    \|\mathcal F_\lambda^G(u) \|_{L^2(T\times G)}^2 &\leq C\Big(\|\mathcal F_\lambda^G(u)\|_{L^2(U\times G)}^2+ \|\mathcal{F}_\lambda^G(Pu)\|_{L^{2}(T \times G)}  \| \mathcal{F}_\lambda^G(u)\|_{L^{2}(T\times G)}\Big).
  \end{align}
  Indeed, recall that $T$ and $G$ have total measure equal to $1$ and the following equality for an arbitrary open set $V\subset T$:
  \begin{align*}
    \|\mathcal F_\lambda^G(u) \|_{L^2(V\times G)}^2 = \left\|\sum_{i=1}^{d_\lambda^G} \mathcal F_\lambda^G(u)_i \otimes \phi_i^\lambda \right\|_{L^2(V\times G)}^2 = \sum_{i=1}^{d_\lambda^G} \|\mathcal F_\lambda^G(u)_i \|_{L^2(V)}^2.
  \end{align*}
  
  Now we use~\eqref{eq:regnumponto} and~\eqref{eq:equcominte} to conclude that for every $\rho\in \R$ we can find $C_\rho>0$ such that
  \begin{align*}
    \|\mathcal F_\lambda^G(u) \|_{L^2(T\times G)}^2 &\leq C_\rho^2(1+\lambda)^{-2\rho}+C\|\mathcal{F}_\lambda^G(Pu)\|_{L^{2}(T \times G)}  \| \mathcal{F}_\lambda^G(u)\|_{L^{2}(T\times G)}\\
    &\leq C_\rho^2(1+\lambda)^{-2\rho}+ \frac{C^2}{2}\|\mathcal{F}_\lambda^G(Pu)\|_{L^{2}(T \times G)}^2 + \frac{1}{2}\|\mathcal{F}_\lambda^G(u)\|_{L^{2}(T\times G)}^2, \quad \forall \lambda \in \sigma(\Delta_G)
  \end{align*}
  from which we conclude, after changing $C_\rho$ if necessary, that
  \begin{align*}
    \|\mathcal F_\lambda^G(u) \|_{L^2(T\times G)}^2 &\leq 2C_\rho^2(1+\lambda)^{-2\rho}+ C^2 \|\mathcal{F}_\lambda^G(Pu)\|_{L^{2}(T \times G)}^2.
  \end{align*}
  Since $Pu \in \cinfty(T\times G)$ we may use~\cite[Corollary~2.9]{afr20} to infer a similar decay in the second term above and conclude that given $\rho>0$, there exists $C_\rho'>0$ such that
  \begin{align}
    \label{Eq:est-fourier-partial-prop'}
    \|\mathcal{F}_\lambda^{G}(u)\|_{L^{2}(T\times G)} &\leq C_\rho'(1+\lambda)^{-\rho}, \quad \forall \lambda \in \sigma(\Delta_G)
  \end{align}
  and now~\eqref{eq:prop-est-logcone} and~\eqref{Eq:est-fourier-partial-prop'} allow us to apply Proposition~\ref{cor:partial_smoothness_converse} and conclude that $u \in \cinfty(T\times G)$.
\end{proof}
\bibliographystyle{plain}
\bibliography{bibliography}
\end{document}